\documentclass[11pt,english]{amsart}

\usepackage{amsmath,amsfonts,amssymb,amsthm,amscd}
\usepackage{dsfont, mathrsfs,enumerate,eucal,fontenc}
\usepackage[varg]{txfonts}
\usepackage{xspace}
\usepackage[english]{babel}

\usepackage{xy, xypic}
\usepackage{color} 
\usepackage{graphicx}
\usepackage{psfrag}  
\usepackage{pst-all}
\usepackage{youngtab}

\usepackage{xcolor}
\definecolor{rouge}{rgb}{0.85,0.1,.4}
\definecolor{bleu}{rgb}{0.1,0.2,0.9}
\definecolor{violet}{rgb}{0.7,0,0.8}
\usepackage[colorlinks=true,linkcolor=bleu,urlcolor=violet,citecolor=rouge]{hyperref}
\usepackage{breakurl}
\usepackage{xspace}

\usepackage{fullpage}

\DeclareMathAlphabet{\mathpzc}{OT1}{pzc}{m}{it}

\theoremstyle{plain}
\newtheorem{theorem}{Theorem}[section]
\newtheorem{lemma}[theorem]{Lemma}
\newtheorem{coro}[theorem]{Corollary}
\newtheorem{prop}[theorem]{Proposition}

\newtheorem{theo}[theorem]{Theorem}
\theoremstyle{definition}

\theoremstyle{remark}

\newtheorem{rema}[theorem]{Remark}

\newtheorem{quest_alpha}{Question}
\def\quest_alpha{\Alph{quest_alpha}}

\def\g{{\mathfrak{g}}}

\def\k{{\Bbbk}}

\def\rg{\ell}               
\def\x{{\rm x}}         
\def\r{{\rm reg}}         
\def\rs{{\rm reg,ss}}          
         
\def\loc{{\mathrm {sm}}}

\def\geq{\geqslant}
\def\leq{\leqslant}

\def\poie#1#2#3#4#5#6#7#8#9{\def\un{#5#6#7#8#9}\def\deux{#6#7#8#9}\def\trois{#2#4#8#9}
\def\quatre{#8#9}\def\cinq{#5#6#7}\def\six{#6#7}\def\sept{#2#4}
\ifx\un\empty {#1}_{#2}{#3 \hskip 0.15em}{#1}_{#4} \else \ifx\deux\empty 
{#5}(#1_{#2}){#3 \hskip 0.15em}{#5}(#1_{#4})
\else \ifx\trois\empty {#5}_{#6}(#1){#3 \hskip 0.15em}{#5}_{#7}(#1) 
\else \ifx\quatre\empty {#5}_{#6}(#1_#2){#3 \hskip 0.15em}{#5}_{#7}(#1_#4) 
\else \ifx\cinq\empty {#1}_{#2}^{#8}{#3 \hskip 0.15em}#1_#4^{#9} 
\else \ifx\six\empty {#5}(#1_{#2}^{#8}){#3 \hskip 0.15em}{#5}(#1_{#4}^{#9}) 
\else \ifx\sept\empty {#5}_{#6}^{#8}(#1){#3 \hskip 0.15em}{#5}_{#7}^{#9}(#1) \else
{#5}_{#6}(#1_{#2}^{#8})^{#9}{#3 \hskip 0.15em}{#5}_{#7}(#1_{#4}^{#8})^{#9} 
\fi \fi \fi \fi \fi \fi \fi}
\def\poi#1#2#3#4#5#6#7{\def\un{#5#6#7}\def\deux{#6#7}
\def\trois{#2#4} \def\cinq{#3#4#5}
\ifx\un\empty {#1}_{#2}{#3 \hskip 0.15em}{#1}_{#4} \else
\ifx\deux\empty {#5}(#1_{#2}){#3 \hskip 0.15em}{#5}(#1_{#4}) \else
\ifx\trois\empty {#5}_{#6}(#1){#3 \hskip 0.15em}{#5}_{#7}(#1) \else
{#5_{#6}}(#1_{#2}){#3 \hskip 0.15em}{#5_{#7}}(#1_{#4}) \fi \fi \fi}
\def\rond{\raisebox{.3mm}{\scriptsize$\circ$}}
\def\mul{\raisebox{.3mm}{\scriptsize\hskip 0.15em$\times$\hskip 0.15em}}
\def\tens{\raisebox{.3mm}{\scriptsize$\otimes$}}

\def\dv#1#2{\langle {#1}\,,{#2}\rangle}

\def\tk#1#2{{#2}\otimes _{#1}}

\def\ec#1#2#3#4#5{\def\un{#3#4#5}\def\deux{#3#5}\def\trois{#3}
\def\four{#2#4#5}\def\five{#2#5}\def\six{#2}\def\seven{#3#4}
\def\eight{#2#4} \def\nine{#2#3#4}
\ifx\nine\empty {\rm #1}_{#5} \else
\ifx\un\empty {\rm #1}({\goth #2}) \else
\ifx\deux\empty {\rm #1}({\goth #2}_{#4}) \else
\ifx\trois\empty {\rm #1}_{#5}({\goth #2}_{#4}) \else
\ifx\four\empty {\rm #1}(#3) \else
\ifx\five\empty {\rm #1}(#3_{#4}) \else
\ifx\six\empty {\rm #1}_{#5}(#3_{#4}) \else
\ifx\seven\empty {\rm #1}_{#5} ({\goth#2})\else
\ifx\eight\empty {\rm #1}_{#5}({#3})
\fi \fi \fi \fi \fi \fi \fi \fi \fi}
\def\hec#1#2#3#4#5{\def\un{#3#4#5}\def\deux{#3#5}\def\trois{#3}
\def\four{#2#4#5}\def\five{#2#5}\def\six{#2}\def\seven{#3#4}
\def\eight{#2#4} \def\nine{#2#3#4}
\ifx\nine\empty \hat{{\rm #1}}_{#5} \else
\ifx\un\empty \hat{{\rm #1}}({\goth #2}) \else
\ifx\deux\empty \hat{{\rm #1}}({\goth #2}_{#4}) \else
\ifx\trois\empty \hat{{\rm #1}}_{#5}({\goth #2}_{#4}) \else
\ifx\four\empty \hat{{\rm #1}}(#3) \else
\ifx\five\empty \hat{{\rm #1}}(#3_{#4}) \else
\ifx\six\empty \hat{{\rm #1}}_{#5}(#3_{#4}) \else
\ifx\seven\empty \hat{{\rm #1}}_{#5} ({\goth#2})  \else
\ifx\eight\empty \hat{{\rm #1}}_{#5}({#3})
\fi \fi \fi \fi \fi \fi \fi \fi \fi}
\def\e#1#2{\ec {#1}#2{}{}{}}
\def\es#1#2{\ec {#1}{}{#2}{}{}}

\def\ai#1#2#3{\def\deux{#2#3} \def\trois{#3} \def\quatre{#2} 
\ifx\deux\empty \es S{{\goth #1}}^{{\goth #1}} \else
\ifx\trois\empty \es S{{\goth #1}^{#2}}^{{\goth #1}^{#2}} \else
\ifx\quatre\empty \es S{{\goth #1}_{#3}}^{{\goth #1}_{#3}} \else
\es S{{\goth #1}_{#3}^{#2}}^{{\goth #1}_{#3}^{#2}} \fi \fi \fi}
\def\aii#1#2#3#4{\def\deux{#2#3} \def\trois{#3} \def\quatre{#2} 
\ifx\deux\empty \sy {#4}{{\goth #1}}^{{\goth #1}} \else
\ifx\trois\empty \sy {#4}{{\goth #1}^{#2}}^{{\goth #1}^{#2}} \else
\ifx\quatre\empty \sy {#4}{{\goth #1}_{#3}}^{{\goth #1}_{#3}} \else
\sy {#4}{{\goth #1}_{#3}^{#2}}^{{\goth #1}_{#3}^{#2}} \fi \fi \fi}

\def\hhom{\mathscr {H}\hskip -.15em om}

\def\Bbb{\mathbb}
\def\goth{\mathfrak}
\def\cal{\mathcal}

\def\gi#1#2#3#4{\def\trois{#3#4} \def\quatre{#4}\def\cinq{#3}\ifx\trois\empty {\rm i}_{#1,{\goth #2}}
\else \ifx\quatre\empty {\rm i}_{#1_{#3},{\goth #2}} \else\ifx\cinq\empty {\rm i}_{#1,{\goth #2}_{#4}} \else {\rm i}_{#1_{#3},{\goth #2}_{#4}} \fi \fi \fi}
\def\j#1#2{\def\deux{#2} \ifx\deux\empty {\rm rk}\hskip .125em{{\goth #1}} \else {\rm rk}\hskip .125em{{\goth #1}_{#2}} \fi}
\def\aj#1#2{\def\deux{#2} \ifx\deux\empty {\rm j}_{{\goth #1}} \else {\rm j}_{{\goth #1}_{#2}} \fi}
\def\an#1#2{\def\deux{#2} \ifx\deux\empty {\cal O}_{#1} \else {\cal O}_{#1,#2} \fi }
\def\han#1#2{\def\deux{#2} \ifx\deux\empty {\hat{{\cal O}}}_{#1} \else {\hat{{\cal O}}}_{#1,#2} \fi }
\def\dim{{\rm dim}\hskip .125em}

\def\dd{{\rm d}}
\def\ad{{\rm ad}\hskip .1em}

\def\n{{\rm n}}
\def\s{{\rm s}}

\def\sy#1#2{{\rm S}^{#1}(#2)}
\def\ex #1#2{\bigwedge ^{#1}#2}

\def\b#1{{\mathrm {b}}_{{\mathfrak{#1}}}}
\def\sqx#1#2{{#1}\times _{B}{#2}}
\def\sqxx#1#2{G^{#1}\times _{B^{#1}}{#2}}

\def\mycom#1#2{\genfrac{}{}{0pt}{}{#1}{#2}}

\begin{document}

\title
[Commuting variety]
{On the generalized commuting varieties of a reductive Lie algebra.}

\author[Jean-Yves Charbonnel]{Jean-Yves Charbonnel}
\address{Jean-Yves Charbonnel, Universit\'e Paris Diderot - CNRS \\
Institut de Math\'ematiques de Jussieu - Paris Rive Gauche\\
UMR 7586 \\ Groupes, repr\'esentations et g\'eom\'etrie \\
B\^atiment Sophie Germain \\ Case 7012 \\ 
75205 Paris Cedex 13, France}
\email{jean-yves.charbonnel@imj-prg.fr}

\author[M. Zaiter]{Mouchira Zaiter}
\address{Mouchira Zaiter, Universit\'e Libanaise Al- Hadath\\
Facult\'e des sciences, branche I\\
Beyrouth Liban}

\email{zaiter.mouchira@hotmail.fr}



\date\today

\begin{abstract}
The generalized commuting and isospectral commuting varieties of a reductive Lie algebra 
have been introduced in a preceding article. In this note, it is proved that their 
normalizations are Gorenstein with rational singularities. Moreover, their 
canonical modules are free of rank $1$. In particular, the usual commuting variety is 
Gorenstein with rational singularities and its canonical module is free of rank $1$.  
\end{abstract}

\maketitle

\tableofcontents

\section{Introduction} \label{i}
In this note, the base field $\k$ is algebraically closed of characteristic $0$, 
${\goth g}$ is a reductive Lie algebra of finite dimension, $\rg$ is its rank,
$\dim {\goth g}=\rg + 2n$ and $G$ is its adjoint group. As usual, ${\goth b}$ denotes a 
Borel subalgebra of ${\goth g}$, ${\goth h}$ a Cartan subalgebra of ${\goth g}$, 
contained in ${\goth b}$, and $B$ the normalizer of ${\goth b}$ in $G$.

\subsection{Main results.} \label{int1}
By definition, for $k\geq 1$, the generalized commuting variety ${\cal C}^{(k)}$ is the 
closure in ${\goth g}^{k}$ of the set of elements whose components are in a same 
Cartan subalgebra. Denoting by ${\cal B}^{(k)}$ the subset of elements of ${\goth g}^{k}$
whose components are in a same Borel subalgebra and by ${\cal B}_{\n}^{(k)}$ its 
normalization, the generalized isospectral commuting variety ${\cal C}_{\x}^{(k)}$ is 
above ${\cal C}^{(k)}$ and under the inverse image of ${\cal C}^{(k)}$ in 
${\cal B}_{\n}^{(k)}$. For $k=2$, ${\cal C}^{(2)}$ is the commuting variety of 
${\goth g}$ and ${\cal C}_{\x}^{(2)}$ is the isospectral commuting variety considered by 
V. Ginzburg in \cite{Gi}. According to \cite[Proposition 5.6]{CZ}, ${\cal C}_{\x}^{(k)}$ 
is an irreducible variety. For studying these varieties, it is very useful to consider 
the closure in the grassmannian $\ec {Gr}g{}{}{\rg}$ of the orbit of ${\goth h}$ under 
the action of $B$ in $\ec {Gr}g{}{}{\rg}$. Denoting by $X$ this variety, $G.X$ is the 
closure of the orbit of ${\goth h}$ under $G$. Let ${\cal E}_{0}$ and ${\cal E}$ be the 
restrictions to $X$ and $G.X$ of the tautological vector bundle over $\ec {Gr}g{}{}{\rg}$
respectively. Denoting by ${\cal E}^{(k)}$ the fiber product over $G.X$ of $k$ copies of 
${\cal E}$, ${\cal E}^{(k)}$ is a subbundle of $G.X\times {\goth g}^{k}$ and 
${\cal C}^{(k)}$ is the image of ${\cal E}^{(k)}$ by the canonical projection 
$\xymatrix{G.X\times {\goth g}^{k}\ar[r] & {\goth g}^{k}}$. Analogously, denoting by 
${\cal E}_{0}^{(k)}$ the restriction of ${\cal E}^{(k)}$ to $X$, the image 
${\goth X}_{0,k}$ of ${\cal E}_{0}^{(k)}$ by the projection 
$\xymatrix{X\times {\goth g}^{k}\ar[r] & {\goth g}^{k}}$ is the 
closure in ${\goth b}^{k}$ of the set of elements whose components are in a same 
Cartan subalgebra. The fiber bundle $\sqx G{{\cal E}_{0}^{(k)}}$ is a vector bundle of 
rank $\rg$ over the fiber bundle $\sqx GX$ over $G/B$. As for ${\cal C}^{(k)}$, there is
a surjective morphism from $\sqx G{{\cal E}_{0}^{(k)}}$ onto ${\cal C}_{\x}^{(k)}$. As a 
matter of fact, the three morphisms:
$$ \xymatrix{ {\cal E}_{0}^{(k)} \ar[rr]^{\tau _{0,k}} && {\goth X}_{0,k}}, \qquad 
\xymatrix{ {\cal E}^{(k)} \ar[rr]^{\tau _{k}} && {\cal C}^{(k)}}, \qquad 
\xymatrix{ \sqx G{{\cal E}_{0}^{(k)}} \ar[rr]^{\tau _{*,k}} && {\cal C}_{\x}^{(k)}}$$
are projective and birational. According to~\cite[Theorem 1.2]{CZ}, $G.X$ is 
smooth in codimension $1$ so that so is ${\cal E}^{(k)}$. By~\cite[Theorem 1.1]{C1},
$X$ is normal and Gorenstein then so are ${\cal E}_{0}^{(k)}$ and 
$\sqx G{{\cal E}_{0}^{(k)}}$. Denoting by $(G.X)_{\n}$ the normalization of
$G.X$, the pullback bundle of ${\cal E}^{(k)}$ over $(G.X)_{\n}$ is the normalization of 
${\cal E}^{(k)}$. Denoting it by ${\cal E}_{\n}^{(k)}$ we have projective birational 
morphisms:
$$ \xymatrix{ {\cal E}_{0}^{(k)} \ar[rr]^{\tau _{\n,0,k}} && 
\widetilde{{\goth X}_{0,k}}}, \qquad 
\xymatrix{ {\cal E}_{\n}^{(k)} \ar[rr]^{\tau _{\n,k}} && 
\widetilde{{\cal C}^{(k)}}}, \qquad 
\xymatrix{ \sqx G{{\cal E}_{0}^{(k)}} \ar[rr]^{\tau _{\n,*,k}} && 
\widetilde{{\cal C}_{\x}^{(k)}}} ,$$
with $\widetilde{{\goth X}_{0,k}}$, $\widetilde{{\cal C}^{(k)}}$, 
$\widetilde{{\cal C}_{\x}^{(k)}}$ the normalizations of ${\goth X}_{0,k}$, 
${\cal C}^{(k)}$, ${\cal C}_{\x}^{(k)}$ respectively. According 
to~\cite[Proposition 4.6]{C1}, for some smooth big open subset $O_{0}$ of 
${\goth X}_{0,k}$, there exists a regular differential form of top degree without zero. 
Moreover, the restriction of $\tau _{0,k}$ to $\tau _{0,k}^{-1}(O_{0})$ is an isomorphism 
onto $O_{0}$. By a simple argument, ${\cal C}^{(k)}$ and ${\cal C}_{\x}^{(k)}$ are smooth
in codimension $1$. Moreover, for some smooth big open subsets $O$ and $O_{*}$ in 
${\cal C}^{(k)}$ and ${\cal C}_{\x}^{(k)}$ respectively, the restrictions
of $\tau _{k}$ and $\tau _{*,k}$ to $\tau _{k}^{-1}(O)$ and $\tau _{*,k}^{-1}(O_{*})$ are 
isomorphisms onto $O$ and $O_{*}$ respectively. The main observation of this note is that 
there are regular differential forms of top degree on $O$ and $O_{*}$ without zero. As a 
result, we have the following theorem:

\begin{theo}\label{tint}
The varieties $\widetilde{{\goth X}_{0,k}}$, $\widetilde{{\cal C}^{(k)}}$, 
$\widetilde{{\cal C}_{\x}^{(k)}}$ are Gorenstein with rational singularities and their 
canonical modules are free of rank $1$. Moreover, $(G.X)_{\n}$ is Gorenstein with 
rational singularities.
\end{theo}

In particular, we give a new proof of a Ginzburg's result~[Theorem 1.3.4]\cite{Gi}. For
$k=2$, ${\cal C}^{(2)}$ is the commuting variety of ${\goth g}$ by \cite{Ric} and 
it is normal by~\cite[Theorem 1.1]{C}. So the commuting variety of ${\goth g}$ is 
Gorenstein with rational singularities and its canonical module is free of rank $1$.
Since $\widetilde{{\goth X}_{0,k}}$ has rational singularities, we get that some 
cohomological groups in positive degree are equal to $0$ and we deduce that 
${\goth X}_{0,k}$ is normal.

This note is organized as follows. In Section~\ref{xv}, the variety ${\cal X}$ is 
introduced and we prove that on the smooth loci of ${\cal X}$ and $\sqx G{{\goth b}}$, 
there are regular differential forms of top degree without zero. In Section~\ref{mv}, we 
recall some results about ${\cal E}$, $X$, $G.X$, $(G.X)_{\n}$. In 
Section~\ref{isc}, we give some results about ${\cal C}^{(k)}$ and ${\cal C}_{\x}^{(k)}$ 
and we prove the main result about regular differential forms of top degree on the smooth
loci of these varieties. As a result, we get the main result of 
the note in Section~\ref{rs}. The goal of Section~\ref{nor} is the normality of 
${\goth X}_{0,k}$. At last, in the appendix, some results are given to prove the 
normality of ${\goth X}_{0,k}$ and Theorem~\ref{tint}.

\subsection{Notations}\label{int2}
$\bullet$ An algebraic variety is a reduced scheme over $\k$ of finite type.

$\bullet$ For $V$ a vector space, its dual is denoted by $V^{*}$ and the augmentation 
ideal of its symmetric algebra $\es SV$ is denoted by $\ec S{}V{}+$. For $A$ a graded
algebra over ${\Bbb N}$, $A_{+}$ is the ideal generated by the homogeneous elements of
positive degree.

$\bullet$
All topological terms refer to the Zariski topology. If $Y$ is a subset of a topological
space $X$, denote by $\overline{Y}$ the closure of $Y$ in $X$. For $Y$ an open subset
of the algebraic variety $X$, $Y$ is called {\it a big open subset} if the codimension
of $X\setminus Y$ in $X$ is at least $2$. For $Y$ a closed subset of an algebraic 
variety $X$, its dimension is the biggest dimension of its irreducible components and its
codimension in $X$ is the smallest codimension in $X$ of its irreducible components. For 
$X$ an algebraic variety, $\an X{}$ is its structural sheaf, $X_{\loc}$ is its smooth 
locus, $\k[X]$ is the algebra of regular functions on $X$ and $\k(X)$ is the field of 
rational functions on $X$ when $X$ is irreducible. When $X$ is smooth and irreducible, 
the sheaf of regular differential forms of top degree on $X$ is denoted by $\Omega _{X}$.

$\bullet$
For $X$ an algebraic variety and for ${\cal M}$ a sheaf on $X$, $\Gamma (V,{\cal M})$
is the space of local sections of ${\cal M}$ over the open subset $V$ of $X$. For 
$i$ a nonnegative integer, ${\mathrm {H}}^{i}(X,{\cal M})$ is the $i$-th group of 
cohomology of ${\cal M}$. For example, 
${\mathrm {H}}^{0}(X,{\cal M})=\Gamma (X,{\cal M})$.

\begin{lemma}\label{lint}~\cite[Corollaire 5.4.3]{Gro}
Let $X$ be an irreducible affine algebraic variety and let $Y$ be a desingularization of 
$X$. Then ${\mathrm {H}}^{0}(Y,\an Y{})$ is the integral closure of $\k[X]$ in its 
fraction field. 
\end{lemma}

$\bullet$
For $E$ a set and $k$ a positive integer, $E^{k}$ denotes its $k$-th cartesian power. If 
$E$ is finite, its cardinality is denoted by $\vert E \vert$. 

$\bullet$ For ${\goth a}$ reductive Lie algebra, its rank is denoted by $\j a{}$ and the
dimension of its Borel subalgebras is denoted by $\b a{}$. In particular, 
$\dim {\goth a}=2\b a{} - \j a{}$.

$\bullet$
If $E$ is a subset of a vector space $V$, denote by span($E$) the vector subspace of
$V$ generated by $E$. The grassmannian of all $d$-dimensional subspaces of $V$ is denoted
by Gr$_d(V)$. By definition, a {\it cone} of $V$ is a subset of $V$ invariant under the 
natural action of $\k^{*}:=\k\setminus \{0\}$ and a \emph{multicone} of $V^{k}$ is a
subset of $V^{k}$ invariant under the natural action of $(\k^{*})^{k}$ on $V^{k}$.

$\bullet$
The dual of ${\goth g}$ is denoted by ${\goth g}^{*}$ and it identifies with ${\goth g}$
by a given non degenerate, invariant, symmetric bilinear form $\dv ..$ on 
${\goth g}\times {\goth g}$, extending the Killing form of $[{\goth g},{\goth g}]$. 

$\bullet$
Let ${\goth b}$ be a Borel subalgebra of ${\goth g}$ and let ${\goth h}$ be a Cartan 
subalgebra of ${\goth g}$ contained in ${\goth b}$. Denote by ${\cal R}$ the root 
system of ${\goth h}$ in ${\goth g}$ and by ${\cal R}_{+}$ the positive root 
system of ${\cal R}$ defined by ${\goth b}$. The Weyl group of ${\cal R}$ is denoted by 
$W({\cal R})$ and the basis of ${\cal R}_{+}$ is denoted by $\Pi $. The neutral
elements of $G$ and $W({\cal R})$ are denoted by $1_{{\goth g}}$ and $1_{{\goth h}}$
respectively. For $\alpha $ in ${\cal R}$, the corresponding root subspace is denoted by 
${\goth g}^{\alpha }$ and a generator $x_{\alpha }$ of ${\goth g}^{\alpha }$ is chosen so
that $\dv {x_{\alpha }}{x_{-\alpha }} = 1$ for all $\alpha $ in ${\cal R}$. Let 
$H_{\alpha }$ be the coroot of $\alpha $.

$\bullet$ The normalizers of ${\goth b}$ and ${\goth h}$ in $G$ are denoted by $B$ and
$N_{G}({\goth h})$ respectively. For $x$ in ${\goth b}$, $\overline{x}$ is the element
of ${\goth h}$ such that $x-\overline{x}$ is in the nilpotent radical ${\goth u}$ of 
${\goth b}$.

$\bullet$ 
For $X$ an algebraic $B$-variety, denote by $\sqx GX$ the quotient of 
$G\times X$ under the right action of $B$ given by $(g,x).b := (gb,b^{-1}.x)$. More 
generally, for $k$ positive integer and for $X$ an algebraic $B^{k}$-variety, denote 
by $\sqxx kX$ the quotient of $G^{k}\times X$ under the right action of $B^{k}$ given by 
$(g,x).b := (gb,b^{-1}.x)$ with $g$ and $b$ in $G^{k}$ and $B^{k}$ respectively. 

\begin{lemma}\label{l2int}
Let $P$ and $Q$ be parabolic subgroups of $G$ such that $P$ is contained in $Q$. Let 
$X$ be a $Q$-variety and let $Y$ be a closed subset of $X$, invariant under $P$. Then 
$Q.Y$ is a closed subset of $X$. Moreover, the canonical map from 
$Q\times _{P}Y$ to $Q.Y$ is a projective morphism.
\end{lemma}

\begin{proof}
Since $P$ and $Q$ are parabolic subgroups of $G$ and since $P$ is contained in $Q$, 
$Q/P$ is a projective variety. Denote by $Q\times _{P}X$ and $Q\times _{P}Y$ the 
quotients of $Q\times X$ and $Q\times Y$ under the right action of $P$ given by 
$(g,x).p := (gp,p^{-1}.x)$. Let $g\mapsto \overline{g}$ be the quotient map from
$Q$ to $Q/P$. Since $X$ is a $Q$-variety, the map 
$$\begin{array}{ccc}
Q\times X \longrightarrow Q/P \times X && (g,x) \longmapsto (\overline{g},g.x)
\end{array}$$
defines through the quotient an isomorphism from $Q\times _{P}X$ to $Q/P\times X$. 
Since $Y$ is a $P$-invariant closed subset of $X$, $Q\times _{P}Y$ is a closed subset
of $Q\times _{P}X$ and its image by the above isomorphism equals $Q/P\times Q.Y$. Hence
$Q.Y$ is a closed subset of $X$ since $Q/P$ is a projective variety. From the commutative
diagram:
$$\xymatrix{ Q\times _{P}Y \ar[r] \ar[rd] & Q/P\times Q.Y \ar[d] \\ & Q.Y }$$
we deduce that the map $\xymatrix{Q\times _{P}Y \ar[r] & Q.Y}$ is a projective morphism.
\end{proof}

$\bullet$
For $k\geq 1$ and for the diagonal action of $B$ in ${\goth b}^{k}$, ${\goth b}^{k}$ is a
$B$-variety. The canonical map from $G\times {\goth b}^{k}$ to $\sqx G{{\goth b}^{k}}$ is
denoted by 
$(g,\poi x1{,\ldots,}{k}{}{}{})\mapsto \overline{(g,\poi x1{,\ldots,}{k}{}{}{})}$. Let 
${\cal B}^{(k)}$ be the image of $G\times {\goth b}^{k}$ by the map 
$(g,\poi x1{,\ldots,}{k}{}{}{})\mapsto (\poi x1{,\ldots,}{k}{g}{}{})$ so that 
${\cal B}^{(k)}$ is a closed subset of ${\goth g}^{k}$ by Lemma~\ref{l2int}. Let 
${\cal B}_{\n}^{(k)}$ be the normalization of ${\cal B}^{(k)}$ and $\eta $ the 
normalization morphism. We have the commutative diagram:
$$ \xymatrix{{\sqx G{{\goth b}^{k}}} \ar[rr]^{\gamma _{\n}} \ar[rd]_{\gamma} && 
{\cal B}_{\n}^{(k)} \ar[ld]^{\eta _{\n}}\\  & {\cal B}^{(k)} &  }  .$$

$\bullet$ Let $i_{k}$ be the injection $(\poi x1{,\ldots,}{k}{}{}{})\mapsto 
\overline{(1_{{\goth g}},\poi x1{,\ldots,}{k}{}{}{})}$ from ${\goth b}^{k}$ to
$\sqx G{{\goth b}^{k}}$. Then $\iota _{k}:= \gamma \rond i_{k}$ and 
$\iota _{\n,k} := \gamma _{\n}\rond i_{k}$
are closed embeddings of ${\goth b}^{k}$ into ${\cal B}^{(k)}$ and ${\cal B}_{\n}^{(k)}$ 
respectively. In particular, ${\cal B}^{(k)} = G.\iota _{k}({\goth b}^{k})$ and 
${\cal B}_{\n}^{(k)} = G.\iota _{\n,k}({\goth b}^{k})$.

$\bullet$
Let $e$ be the sum of the $x_{\beta }$'s, $\beta $ in $\Pi $, and let $h$ be the 
element of ${\goth h}\cap [{\goth g},{\goth g}]$ such that $\beta (h)=2$ for all $\beta $
in $\Pi $. Then there exists a unique $f$ in $[{\goth g},{\goth g}]$ such that $(e,h,f)$ 
is a principal ${\goth {sl}}_2$-triple. The one parameter subgroup of $G$ generated by 
$\ad h$ is denoted by $t\mapsto h(t)$. The Borel subalgebra containing $f$ is denoted by 
${\goth b}_{-}$ and its nilpotent radical is denoted by ${\goth u}_{-}$. Let $B_{-}$ be 
the normalizer of ${\goth b}_{-}$ in $G$ and let $U$ and $U_{-}$ be the unipotent 
radicals of $B$ and $B_{-}$ respectively.

\begin{lemma}\label{l3int}
Let $k\geq 2$ be an integer. Let $X$ be an affine variety and set 
$Y:= {\goth b}^{k}\times X$. Let $Z$ be a closed subset of $Y$ invariant under the 
action of $B$ given by 
$g.(\poi v1{,\ldots,}{k}{}{}{},x)=(\poi v1{,\ldots,}{k}{g}{}{},x)$ with 
$(g,\poi v1{,\ldots,}{k}{}{}{})$ in $B\times {\goth b}^{k}$ and $x$ in $X$. 
Then $Z\cap {\goth h}^{k}\times X$ is the image of $Z$ by the projection 
$(\poi v1{,\ldots,}{k}{}{}{},x)\mapsto (\overline{v_{1}},\ldots,\overline{v_{k}},x)$.  
\end{lemma}

\begin{proof}
For all $v$ in ${\goth b}$, 
$$ \overline{v} = \lim _{t\rightarrow 0} h(t)(v)$$
whence the lemma since $Z$ is closed and $B$-invariant.
\end{proof}

$\bullet$ 
For $x \in \g$, let $x_{\s}$ and $x_{\n}$ be the semisimple and nilpotent components of 
$x$ in ${\goth g}$. Denote by ${\goth g}^x$ and $G^{x}$ the centralizers of $x$ in 
${\goth g}$ and $G$ respectively. For ${\goth a}$ a subalgebra of ${\goth g}$ and for $A$
a subgroup of $G$, set:
$$\begin{array}{ccc}
{\goth a}^{x} := {\goth a}\cap {\goth g}^{x} && A^{x} := A \cap G^{x} \end{array} .$$
The set of regular elements of $\g$ is 
$$\g_{\r} \ := \ \{ x\in \g \ \vert \ \dim \g^x=\rg \} .$$
Denote by ${\goth g}_{\rs}$ the set of regular semisimple elements of
${\goth g}$. Both $\g_{\r}$ and $\g_{\rs}$ are $G$-invariant dense open subsets of
${\goth g}$. Setting ${\goth h}_{\r} := {\goth h}\cap {\goth g}_{\r}$, 
${\goth b}_{\r} := {\goth b}\cap {\goth g}_{\r}$, 
${\goth g}_{\rs}=G({\goth h}_{\r})$ and ${\goth g}_{\r}=G({\goth b}_{\r})$. 

$\bullet$
Let $\poi p1{,\ldots,}{\rg}{}{}{}$ be some homogeneous polynomials generating the algebra
$\e Sg^{G}$ of invariant polynomials under $G$. For $i =1,\ldots,\rg$ and for  $x$ in 
$\g$, denote by $\varepsilon _i(x)$ the element of $\g$ given by
$$ \dv {\varepsilon _{i}(x)}y = \frac{\dd }{\dd t} p_{i}(x+ty) \left \vert _{t=0} \right.
$$
for all $y$ in ${\goth g}$. Thereby, $\varepsilon _{i}$ is an invariant element of 
$\tk {\k}{\e Sg}\g$ under the canonical action of $G$. According to 
\cite[Theorem 9]{Ko}, for $x$ in ${\goth g}$, $x$ is in ${\goth g}_{\r}$ if and only if 
$\poi x{}{,\ldots,}{}{\varepsilon }{1}{\rg}$ are linearly independent. In this case, 
$\poi x{}{,\ldots,}{}{\varepsilon }{1}{\rg}$ is a basis of ${\goth g}^{x}$. 

\section{On the varieties \texorpdfstring{${\cal X}$}{Lg} and 
\texorpdfstring{$\sqx G{{\goth b}}$}{Lg}} \label{xv}
Denote by $\pi _{{\goth g}} : {\goth g}\rightarrow {\goth g}//G$ and 
$\pi _{{\goth h}} : {\goth h} \rightarrow {\goth h}/W({\cal R)}$ the quotient maps,
i.e the morphisms defined by the invariants. Recall ${\goth g}//G={\goth h}/W({\cal R})$,
and let ${\cal X}$ be the following fiber product:
$$ \xymatrix{ {\cal X} \ar[rr]^{\overline{\chi }} \ar[d]_{\overline{\rho }} && 
{\goth g} \ar[d]^{\pi _{{\goth g}}} \\ {\goth h} \ar[rr]_{\pi _{{\goth h}}} &&
{\goth h}/W({\cal R)}} $$
where $\overline{\chi }$ and $\overline{\rho }$ are the restriction maps. The 
actions of $G$ and $W({\cal R})$ on ${\goth g}$ and ${\goth h}$ respectively induce
an action of $G\times W({\cal R})$ on ${\cal X}$. According to \cite[Lemma 2.4]{CZ}, 
${\cal X}$ is irreducible and normal. Moreover, 
${\cal X}_{\r} := {\goth g}_{\r}\times {\goth h}\cap {\cal X}$ is a smooth open
subset of ${\cal X}$, $\k[{\cal X}]$ is the space of global sections 
$\an {\sqx G{{\goth b}}}{}$ and $\k[{\cal X}]^{G}=\e Sh$. According to
\cite[Lemma 2.4]{CZ}, the map
$$ \xymatrix{ G\times {\goth b} \ar[rr] && {\cal X} }, \qquad 
(g,x) \longmapsto (g(x),\overline{x}) $$
defines through the quotient a projective birational morphism 
$$ \xymatrix{ \sqx G{{\goth b}} \ar[rr]^{\chi _{\n}} && {\cal X}} .$$ 

\begin{lemma}\label{lxv}
{\rm (i)} The set ${\goth b}_{\r}$ is a big open subset of ${\goth b}$.

{\rm (ii)} The set $\sqx G{{\goth b}_{\r}}$ is a big open subset of $\sqx G{{\goth b}}$.

{\rm (iii)} The restriction of $\chi _{\n}$ to $\sqx G{{\goth b}_{\r}}$ is an 
isomorphism onto ${\cal X}_{\r}$.

{\rm (iv)} The restriction of $\pi _{{\goth g}}$ to ${\goth g}_{\r}$ is a smooth 
morphism. 
\end{lemma}

\begin{proof}
(i) Let $\Sigma $ be an irreducible component of ${\goth b}\setminus {\goth b}_{\r}$.
Then $\Sigma $ is a closed cone invariant under $B$ and 
$\overline{\Sigma } := \Sigma \cap {\goth h}$ is a closed cone of ${\goth h}$. According
to Lemma~\ref{l3int}, $\Sigma $ is contained in $\overline{\Sigma }+{\goth u}$. Suppose 
that $\Sigma $ has codimension $1$ in ${\goth b}$. A contradiction is expected. Then
$\overline{\Sigma }={\goth h}$ or $\overline{\Sigma }$ has codimension $1$ in ${\goth h}$.
The first case is impossible since ${\goth h}\cap {\goth b}_{\r}$ is not empty. Hence 
$\Sigma = \overline{\Sigma }+{\goth u}$ since $\Sigma $ is irreducible of codimension 
$1$ in ${\goth b}$. As a result, ${\goth u}$ is contained in $\Sigma $ since 
$\overline{\Sigma }$ is a closed cone, whence the contradiction since 
${\goth u}\cap {\goth b}_{\r}$ is not empty.

(ii) The complement of $\sqx G{{\goth b}_{\r}}$ in $\sqx G{{\goth b}}$ is equal to
$\sqx G{{\goth b}\setminus {\goth b}_{\r}}$. By (i), ${\goth b}\setminus {\goth b}_{\r}$,
is a $B$-invariant closed subset of ${\goth b}$ of dimension at most $\dim {\goth b}-2$.
Then $\sqx G{{\goth b}\setminus {\goth b}_{\r}}$ is a closed subset of 
$\sqx G{{\goth b}}$ of codimension at least $2$, whence the assertion.

(iii) By definition, ${\cal X}_{\r}=\chi _{\n}(\sqx G{{\goth b}_{\r}})$. Let 
$(g_{1},x_{1})$ and $(g_{2},x_{2})$ be in $G\times {\goth b}_{\r}$ such that 
$(g_{1}(x_{1}),\overline{x_{1}})=(g_{2}(x_{2}),\overline{x_{2}})$. For some $b_{1}$ and 
$b_{2}$ in $B$, 
$$b_{1}(x_{1})_{\s} = \overline{x_{1}} \quad  \text{and} \quad 
b_{2}(x_{2})_{\s} = \overline{x_{2}} = \overline{x_{1}}.$$
Setting:
$$ y_{1} := b_{1}(x_{1}) \quad  \text{and} \quad y_{2} := b_{2}(x_{2}) ,$$
$y_{2} = b_{2}g_{2}^{-1}g_{1}b_{1}^{-1}(y_{1})$ is a regular element of 
${\goth g}^{\overline{x_{1}}}$. In particular, $y_{2,\n}$ and $y_{1,\n}$ are regular 
nilpotent elements of ${\goth g}^{\overline{x_{1}}}$ and they are in the borel subalgebra
${\goth b}\cap {\goth g}^{\overline{x_{1}}}$ of ${\goth g}^{\overline{x_{1}}}$. Hence 
$b_{2}g_{2}^{-1}g_{1}b_{1}^{-1}$ is in $B$ and so is $g_{2}^{-1}g_{1}$. As a result, 
the restriction of $\chi _{\n}$ to $\sqx G{{\goth b}_{\r}}$ is injective. So, 
by Zariski's Main Theorem \cite[\S 9]{Mu}, the restriction of $\chi _{\n}$ to 
$\sqx G{{\goth b}_{\r}}$ is an isomorphism onto ${\cal X}_{\r}$ since ${\cal X}_{\r}$ is
a smooth variety.

(iv) Let $x$ be in ${\goth g}_{\r}$. The kernel of the differential of $\pi _{{\goth g}}$
at $x$ is the orthogonal complement of ${\goth g}^{x}$ so that the differential of 
$\pi _{{\goth g}}$ at $x$ is surjective whence the assertion by 
\cite[Ch. III, Proposition 10.4]{Ha}.
\end{proof}

\begin{prop}\label{pxv}
{\rm (i)} There exists a regular form of top degree, without zero on ${\cal X}_{\r}$.

{\rm (ii)} There exists a regular form of top degree, without zero on $\sqx G{{\goth b}}$.
\end{prop}

\begin{proof}
(i) Let $\omega $ be a volume form on ${\goth g}$. According to 
Lemma~\ref{lxv},(iv), the restriction of $\omega $ to ${\goth g}_{\r}$ is divisible by 
$\poi {\dd p}1{\wedge \cdots \wedge }{\rg}{}{}{}$ so that
$$ \omega  = \alpha \wedge \poi {\dd p}1{\wedge \cdots \wedge }{\rg}{}{}{}$$
with $\alpha $ a regular relative differential form of top degree with respect to 
$\pi _{{\goth g}}$. Denoting by $\poi v1{,\ldots,}{\rg}{}{}{}$ a basis of ${\goth h}$, 
$$ \omega ' := \alpha \wedge \poi {\dd v}1{ \wedge \cdots \wedge }{\rg}{}{}{}$$
is a regular form of top degree on ${\cal X}_{\r}$ since $\e Sg^{G}$ identifies 
with a subalgebra of $\e Sh$. As $\pi _{{\goth g}}$ and $\overline{\rho }$ have the same
fibers and $\omega $ has no zero so has $\omega '$.

(ii) By Lemma~\ref{lxv},(iii), $\chi _{\n}^{*}(\omega ')$ is a regular form of 
top degree on $\sqx G{{\goth b}_{\r}}$ without zero. Then by Lemma~\ref{lars},(ii) and 
Lemma~\ref{lxv},(ii), theres exists a regular form of top degree on $\sqx G{{\goth b}}$, 
without zero.
\end{proof}

\section{Main varieties and tautological vector bundles} \label{mv}
Denote by $X$ the closure in $\ec {Gr}g{}{}{\rg}$ of the orbit of ${\goth h}$ under 
$B$. Since $G/B$ is a projective variety, $G.X$ is the closure in $\ec {Gr}g{}{}{\rg}$ of
the orbit of ${\goth h}$ under $G$. Set:
$$ {\cal E}_{0} := \{(u,x) \in X \times {\goth b} \; \vert ; x \in u\}, \qquad
{\cal E} := \{(u,x) \in G.X \times {\goth g} \; \vert ; x \in u\} .$$ 
Then ${\cal E}_{0}$ and ${\cal E}$ are the restrictions to $X$ and $G.X$ respectively of 
the tautological vector bundle of rank $\rg$ over $\ec {Gr}g{}{}{\rg}$. Denote by 
$\pi _{0}$ and $\pi $ the bundle projections:
$$ \xymatrix{ {\cal E}_{0} \ar[rr]^{\pi _{0}} && X}, \qquad 
\xymatrix{ {\cal E} \ar[rr]^{\pi } && G.X} .$$
Since the map
$$ \xymatrix{ {\goth g}_{\r} \ar[rr] && \ec {Gr}g{}{}{\rg}}, \qquad 
x \longmapsto {\goth g}^{x}$$
is regular, for all $x$ in ${\goth g}_{\r}$, ${\goth g}^{x}$ is in $G.X$ and for all
$x$ in ${\goth b}_{\r}$, ${\goth g}^{x}$ is in $X$. Denoting by $X'$ the image of 
${\goth b}_{\r}$, $G.X'$ is the image of ${\goth g}_{\r}$ and according to 
\cite[Theorem 1.2]{CZ}, $X'$ and $G.X'$ are smooth big open subsets of $X$ and $G.X$ 
respectively.

Let $\tau _{0}$ and $\tau $ be the restrictions to ${\cal E}_{0}$ and ${\cal E}$ 
respectively of the canonical projection 
$\ec {Gr}g{}{}{\rg}\times {\goth g} \rightarrow {\goth g}$. Denote by $\pi _{*}$ and 
$\tau _{*}$ the morphisms
$$ \xymatrix{\sqx G{{\cal E}_{0}} \ar[rr]^{\pi _{*}} && \sqx GX}, \qquad  \text{and} \quad
\xymatrix{\sqx G{{\cal E}_{0}} \ar[rr]^{\tau _{*}} && {\cal X}}$$
defined through the quotients by the maps 
$$ \xymatrix{ G\times {\cal E}_{0} \ar[rr] && G\times X}, \qquad 
(g,u,x) \longmapsto (g,u) ,$$
$$ \xymatrix{ G\times {\cal E}_{0} \ar[rr] && {\cal X}}, \qquad 
(g,u,x) \longmapsto (g(x),\overline{x}) .$$

\begin{lemma}\label{lmv}
{\rm (i)} The morphism $\tau _{0}$ is a projective and birational morphism from 
${\cal E}_{0}$ onto ${\goth b}$.

{\rm (ii)} The morphism $\tau $ is a projective and birational morphism from ${\cal E}$ 
onto ${\goth g}$. 

{\rm (iii)} The morphism $\tau _{*}$ is a projective and birational morphism from 
$\sqx G{{\cal E}_{0}}$ onto ${\cal X}$.
\end{lemma}

\begin{proof}
(i) and (ii) Since $X$ and $G.X$ are projective varieties, $\tau _{0}$ and $\tau $ are 
projective morphisms. For $x$ in ${\goth g}_{\r}$, $\tau ^{-1}(x) = \{{\goth g}^{x}\}$. 
Hence $\tau _{0}$ and $\tau $ are birational and their images are ${\goth b}$ and 
${\goth g}$ since ${\goth g}_{\r}$ is an open subset of ${\goth g}$.

(iii) The morphism 
$$ \xymatrix{ G\times {\cal E}_{0} \ar[rr] && G\times {\goth b}}, \qquad
(g,u,x) \longmapsto (g,x)$$
defines through the quotient a morphism 
$$ \xymatrix { \sqx G{{\cal E}_{0}} \ar[rr]^{\tau _{1}} && \sqx G{{\goth b}}} .$$
The varieties $\sqx G{{\cal E}_{0}}$ and $\sqx G{{\goth b}}$ are embedded into 
$G/B\times {\cal E}$ and $G/B\times {\goth g}$ respectively as closed subsets and 
$\tau _{1}$ is the restriction to $\sqx G{{\cal E}_{0}}$ of 
${\mathrm {id}}_{G/B}\mul \tau $. Hence $\tau _{1}$ is a projective
morphism by (ii). As $\tau _{*}$ is the composition of $\tau _{1}$ and $\chi _{\n}$, 
$\tau _{*}$ is a projective morphism since so is $\chi _{\n}$. The map
$$ \xymatrix{ G\times {\goth b}_{\r} \ar[rr] && G\times {\cal E}_{0}}, \qquad 
(g,x) \longmapsto (g,{\goth g}^{x},x)$$
defines through the quotient a morphism 
$$ \xymatrix{ \sqx G{{\goth b}_{\r}} \ar[rr]^{\mu } && \sqx G{{\cal E}_{0}}} .$$
According to Lemma~\ref{lxv},(iii), the restriction of $\tau _{*}$ to 
$\tau _{*}^{-1}({\cal X}_{\r})$ is an isomorphism onto ${\cal X}_{\r}$ whose inverse is
$ \mu \rond \chi _{\n}^{-1}$. In particular, $\tau _{*}$ is birational.
\end{proof}

Denote by $(G.X)_{\n}$ the normalization of $G.X$. Let ${\cal E}_{\n}$ be the following 
fiber product:
$$ \xymatrix{ {\cal E}_{\n} \ar[rr] ^{\nu _{\n}} \ar[d]_{\pi_{\n}}&& 
{\cal E} \ar[d]^{\pi } \\ (G.X)_{\n} \ar[rr]_{\nu } && G.X} $$
with $\nu $ the normalization morphism, $\nu _{\n}$, $\pi _{\n}$ the restriction maps.

\begin{prop}\label{pmv}
{\rm (i)} The varieties ${\cal E}_{0}$ and $X$ are Gorenstein with rational 
singularities.

{\rm (ii)} The varieties ${\cal E}_{\n}$ and $X_{\n}$ are Gorenstein with rational 
singularities.

{\rm (iii)} The varieties $\sqx G{{\cal E}_{0}}$ and $\sqx G{X}$ are Gorenstein 
with rational singularities.
\end{prop}

\begin{proof}
According to~\cite[Theorem 1.1]{C1}, $X$ is Gorenstein with rational singularities, then
by Lemma~\ref{lsi},(i) and (iv), so is ${\cal E}_{0}$ as a vector bundle over $X$. 
Furthermore, by Lemma~\ref{lsi},(i) and (iii), $\sqx GX$ is Gorenstein with rational 
singularities as a fiber bundle over a smooth variety whose fibers are Gorenstein with 
rational singularities. As a result, by Lemma~\ref{lsi},(i) and (iv), 
$\sqx G{{\cal E}_{0}}$ is Gorenstein with rational singularities as a vector bundle over 
$\sqx GX$.

Proposition~\ref{pmv},(ii) will be proved in Section~\ref{rs} (see Corollary~\ref{crs}).
\end{proof}

\section{On the generalized isospectral commuting variety} \label{isc}
Let $k\geq 2$ be an integer. The variety $\sqxx k{{\goth b}^{k}}$ identifies 
with $(\sqx G{\goth b})^{k}$. Denote by $\chi _{\n}^{(k)}$ the morphism 
$$ \xymatrix{ \sqxx k{{\goth b}^{k}} \ar[rr]^{\chi _{\n}^{(k)}} && {\cal X}^{(k)}},
\qquad (\poi x1{,\ldots,}{k}{}{}{}) \longmapsto 
(\poi x1{,\ldots,}{k}{\chi }{\n}{\n}) .$$
The varitey $G/B$ identifies with the diagonal $\Delta $ of $(G/B)^{k}$ so that 
$\sqx G{{\goth b}^{k}}$ identifies with the restriction to $\Delta $ of the vector
bundle $\sqxx k{{\goth b}^{(k)}}$ over $G/B$. Denote by $\gamma _{\x}$ the restriction of
$\chi _{\n}^{(k)}$ to $\sqx G{{\goth b}^{k}}$ and by ${\cal B}_{\x}^{(k)}$ its image, 
whence a commutative diagram
$$\xymatrix{\sqx G{{\goth b}^{k}} \ar[rr]^{\gamma _{\x}} \ar[rd]_{\gamma }
&& {\cal B}_{\x}^{(k)} \ar[ld]^{\eta } \\ & {\cal B}^{(k)} & }$$ 
with $\eta $ the restriction to ${\cal B}_{\x}^{(k)}$ of the canonical projection
$\xymatrix{{\cal X}^{k} \ar[r] & {\goth g}^{k}}$. 
Let $\iota _{\x,k}$ be the map given by
$$ \xymatrix{ {\goth b}^{k} \ar[rr]^{\iota _{\x,k}} && {\cal X}^{k}}, \qquad
(\poi x1{,\ldots,}{k}{}{}{}) \longmapsto (\poi x1{,\ldots,}{k}{}{}{},
\overline{x_{1}},\ldots,\overline{x_{k}}) .$$
According to \cite[Lemma 2.7,(i) and Corollary 2.8,(i)]{CZ}, $\iota _{\x,k}$ is a closed 
embedding of ${\goth b}^{k}$ into ${\cal B}_{\x}^{(k)}$ and $\gamma _{\x}$ is a 
projective birational morphism so that ${\cal B}_{\n}^{(k)}$ is the normalization of 
${\cal B}_{\x}^{(k)}$. Denote by ${\cal C}^{(k)}$ the closure of $G.{\goth h}^{k}$ in 
${\goth g}^{k}$ with respect to the diagonal action of $G$ in ${\goth g}^{k}$ and set 
${\cal C}_{\x}^{(k)} := \eta ^{-1}({\cal C}^{(k)})$. The 
varieties ${\cal C}^{(k)}$ and ${\cal C}_{\x}^{(k)}$ are called 
{\it generalized commuting variety} and {\it generalized isospectral commuting variety} 
respectively. For $k=2$, ${\cal C}_{\x}^{(k)}$ is the isospectral commuting variety 
considered by M. Haiman in~\cite[\S 8]{Ha1} and~\cite[\S 7.2]{Ha2}. According 
to~\cite[Proposition 5.6]{CZ}, ${\cal C}_{\x}^{(k)}$ is irreducible and equal to the 
closure of $G.\iota _{\x,k}({\goth h}^{k})$ in ${\cal B}_{\x}^{(k)}$.

\subsection{} \label{isc1}
We consider the diagonal action of $B$ in ${\goth b}^{k}$. Let ${\goth X}_{0,k}$ be the 
closure of $B.{\goth h}^{k}$ in ${\goth b}^{k}$. Set:
$$ {\cal E}^{(k)} := \{(u,\poi x1{,\ldots,}{k}{}{}{}) \in G.X\times {\goth g}^{\k} 
\ \vert \ x_{1}\in u,\ldots,x_{k}\in u \} \quad  \text{and} \qquad
{\cal E}_{0}^{(k)} := {\cal E}^{(k)} \cap X\times {\goth b}^{k}.$$
Then ${\cal E}_{0}^{(k)}$ and ${\cal E}^{(k)}$ are vector bundles over $X$ and $G.X$ 
respectively. Denote by $\pi _{0,k}$ and $\pi _{k}$ respectively their bundle 
projections. Let $\tau _{0,k}$ and $\tau _{k}$ be the restrictions to 
${\cal E}_{0}^{(k)}$ and ${\cal E}^{(k)}$ respectively of the canonical projection 
$\ec {Gr}g{}{}{\rg}\times {\goth g}^{k}\rightarrow {\goth g}^{k}$. Denote
by $\pi _{*,k}$ and $\tau _{*,k}$ the morphisms
$$ \xymatrix{\sqx G{{\cal E}_{0}^{(k)}} \ar[rr]^{\pi _{*,k}} && \sqx GX}, 
\qquad  \text{and} \qquad 
\xymatrix{\sqx G{{\cal E}_{0}^{(k)}} \ar[rr]^{\tau _{*,k}} && {\cal X}^{k}}$$
defined through the quotients by the maps 
$$ \xymatrix{ G\times {\cal E}_{0}^{(k)} \ar[rr] && G\times X}, 
\qquad (g,u,x) \longmapsto (g,u) ,$$
$$ \xymatrix{ G\times {\cal E}_{0}^{(k)} \ar[rr] && {\cal X}^{k}}, \qquad 
(g,u,\poi x1{,\ldots,}{k}{}{}{}) \longmapsto (\poi x1{,\ldots,}{k}{g}{}{},
\overline{x_{1}},\ldots,\overline{x_{k}}) .$$

\begin{lemma}\label{lisc1}
{\rm (i)} The morphism $\tau _{0,k}$ is a projective morphism from ${\cal E}_{0}^{(k)}$ 
onto ${\goth X}_{0,k}$.

{\rm (ii)} The morphism $\tau _{k}$ is a projective morphism from ${\cal E}^{(k)}$ onto 
${\cal C}^{(k)}$. 

{\rm (iii)} The morphism $\tau _{*,k}$ is a projective morphism from 
$\sqx G{{\cal E}_{0}^{(k)}}$ onto ${\cal C}_{\x}^{(k)}$.
\end{lemma}

\begin{proof}
(i) Since $X$ is a projective variety, $\tau _{0,k}$ is a projective morphism. Then  
its image is an irreducible closed subset of ${\goth b}^{k}$ since ${\cal E}_{0}^{(k)}$ 
is irreducible as a vector bundle over an irreducible variety. Moreover, $B.{\goth h}^{k}$
is contained in $\tau _{0,k}({\cal E}_{0}^{(k)})$ since $\tau _{0,k}({\cal E}_{0}^{(k)})$
is invariant under $B$ and contains ${\goth h}^{k}$. As a vector bundle of rank $k\rg$ 
over $X$, ${\cal E}_{0}^{(k)}$ has dimension $k\rg + \dim {\goth u}$. Since the 
restriction to $U\times {\goth h}_{\r}^{k}$ of the map
$$ \xymatrix{B\times {\goth h}^{k} \ar[rr] && {\goth b}^{k}}, \qquad
(g,\poi x1{,\ldots,}{k}{}{}{}) \longmapsto (\poi x1{,\ldots,}{k}{g}{}{})$$
is injective, ${\goth X}_{0,k}$ has dimension $\dim u +k\rg$. Hence ${\goth X}_{0,k}$
is the image of ${\cal E}_{0}^{(k)}$ by $\tau _{0,k}$.  

(ii) Since $G.X$ is a projective variety, $\tau _{k}$ is a projective morphism. Then  
its image is an irreducible closed subset of ${\goth g}^{k}$ since ${\cal E}^{(k)}$ is 
irreducible as a vector bundle over an irreducible variety. Moreover, $G.{\goth h}^{k}$
is contained in $\tau _{k}({\cal E}^{(k)})$ since $\tau _{k}({\cal E}^{(k)})$ is 
invariant under $G$ and contains ${\goth h}^{k}$. As a vector bundle of rank $k\rg$ over 
$G.X$, ${\cal E}^{(k)}$ has dimension $k\rg + 2\dim {\goth u}$. Since the fibers of the 
restriction to $G \times {\goth h}_{\r}^{k}$ of the map
$$ \xymatrix{G\times {\goth h}^{k} \ar[rr] && {\goth g}^{k}}, \qquad
(g,\poi x1{,\ldots,}{k}{}{}{}) \longmapsto (\poi x1{,\ldots,}{k}{g}{}{})$$
have dimension $\rg$, ${\cal C}^{(k)}$ has dimension $2\dim {\goth u} + k\rg$. Hence 
${\cal C}^{(k)}$ is the image of ${\cal E}^{(k)}$ by $\tau _{k}$.  

(iii) The morphism 
$$ \xymatrix{ G\times {\cal E}_{0}^{(k)} \ar[rr] && G\times {\goth b}^{k}}, \qquad
(g,u,x) \longmapsto (g,x)$$
defines through the quotient a morphism 
$$ \xymatrix { \sqx G{{\cal E}_{0}^{(k)}} \ar[rr]^{\tau _{1,k}} && 
\sqx G{{\goth b}^{k}}} .$$
The varieties $\sqx G{{\cal E}_{0}^{(k)}}$ and $\sqx G{{\goth b}^{k}}$ are embedded into 
$G/B\times {\cal E}^{(k)}$ and $G/B\times {\goth g}^{k}$ respectively as closed subsets 
and $\tau _{1,k}$ is the restriction to $\sqx G{{\cal E}_{0}^{(k)}}$ of 
${\mathrm {id}}_{G/B}\mul \tau _{k}$. Hence $\tau _{1,k}$ is a projective
morphism by (ii). As $\tau _{*,k}$ is the composition of $\tau _{1,k}$ and 
$\gamma _{\x}$, $\tau _{*,k}$ is a projective morphism since so is $\gamma _{\x}$. 
Moreover, by (ii), the image of $\eta \rond \tau _{*,k}$ is equal to 
${\cal C}^{(k)}$. Hence ${\cal C}_{\x}^{(k)}$ is the image of $\tau _{*,k}$ since it is
irreducible and equal to $\eta ^{-1}({\cal C}^{(k)})$.
\end{proof}

\subsection{} \label{isc2} 
For $j=1,\ldots,k$, denote by $V_{0,j}^{(k)}$ the subset of elements of 
${\goth X}_{0,k}$ whose $j$-th component is in ${\goth b}_{\r}$ and by $V_{j}^{(k)}$ the 
subset of elements of ${\cal C}^{(k)}$ whose $j$-th component is in ${\goth g}_{\r}$. Let
$W_{j}^{(k)}$ be the inverse image of $V_{j}^{(k)}$ by $\eta $. 

Let $\sigma _{j}$ be the automorphism of ${\goth g}^{k}$ permuting the first and the 
$j$-th components of its elements. Then $\sigma _{j}$ is equivariant under the diagonal 
action of $G$ in ${\goth g}^{k}$ and ${\goth b}^{k}$ and ${\goth h}^{k}$ are invariant 
under $\sigma _{j}$. As a result, ${\goth X}_{0,k}$ is invariant under $\sigma _{j}$ and 
$\sigma _{j}(V_{0,1}^{(k)}) = V_{0,j}^{(k)}$. In the same way, ${\cal C}^{(k)}$ is 
invariant under $\sigma _{j}$ and $\sigma _{j}(V_{1}^{(k)}) = V_{j}^{(k)}$. The map
$$ \xymatrix{ G\times {\goth b}^{k} \ar[rr] && G\times {\goth b}^{k}}, \qquad
(g,x) \longmapsto (g,\sigma _{j}(x))$$
defines through the quotient an automorphism of $\sqx G{{\goth b}^{k}}$. Denote again by 
$\sigma _{j}$ this automorphism and the restriction to ${\cal X}^{k}$ of the automorphism 
$(x,y)\mapsto (\sigma _{j}(x),\sigma _{j}(y))$ of ${\goth g}^{k}\times {\goth h}^{k}$. 
Since ${\cal B}_{\x}^{(k)}$ is contained in ${\cal X}^{k}$ and $\gamma _{\x}$ is a 
morphism from $\sqx G{{\goth b}^{k}}$ to ${\cal X}^{k}$ such that 
$\gamma _{\x}\rond \sigma _{j} = \sigma _{j}\rond \gamma _{\x}$, 
${\cal B}_{\x}^{(k)}$ is invariant under $\sigma _{j}$. In the same way, 
$\sigma _{j}\rond \gamma = \gamma \rond \sigma _{j}$ and ${\cal B}^{(k)}$ is invariant 
under $\sigma _{j}$. As a result $\sigma _{j}\rond \eta = \eta \rond \sigma _{j}$, 
${\cal C}_{\x}^{(k)}$ is invariant under $\sigma _{j}$ and 
$\sigma _{j}(W_{1}^{(k)})=W_{j}^{(k)}$. 

\begin{lemma}\label{lisc2}
Let $j=1,\ldots,k$.

{\rm (i)} The set $V_{0,j}^{(k)}$ is a smooth open subset of ${\goth X}_{0,k}$. Moreover
there exists a regular differential form of top degree on $V_{0,j}^{(k)}$, without zero.

{\rm (ii) } The set $V_{j}^{(k)}$ is a smooth open subset of ${\cal C}^{(k)}$. Moreover
there exists a regular differential form of top degree on $V_{j}^{(k)}$, without zero.

{\rm (iii)} The set $W_{j}^{(k)}$ is a smooth open subset of ${\cal C}_{\x}^{(k)}$. 
Moreover there exists a regular differential form of top degree on $W_{j}^{(k)}$, without
zero.
\end{lemma}

\begin{proof}
According to the above remarks, we can suppose $j=1$.

(i) By definition, $V_{0,1}^{(k)}$ is the intersection of ${\goth X}_{0,k}$ and the open 
subset ${\goth b}_{\r}\times {\goth b}^{k-1}$ of ${\goth b}^{k}$. Hence $V_{0,1}^{(k)}$
is an open subset of ${\goth X}_{0,k}$. For $x_{1}$ in ${\goth b}_{\r}$, 
$(\poi x1{,\ldots,}{k}{}{}{})$ is in $V_{0,1}^{(k)}$ if and only if 
$\poi x2{,\ldots,}{k}{}{}{}$ are in ${\goth g}^{x_{1}}$ by Lemma~\ref{lisc1},(i) since 
${\goth g}^{x_{1}}$ is in $X$. According to~\cite[Theorem 9]{Ko}, for $x$ in 
${\goth b}_{\r}$, $\poi x{}{,\ldots,}{}{\varepsilon }{1}{\rg}$ is a basis of 
${\goth g}^{x}$ and ${\goth g}^{x}$ is contained in ${\goth b}$. Hence the map 
$$ \begin{array}{ccc} \xymatrix{{\goth b}_{\r}\times {\mathrm {M}_{k-1,\rg}}(\k)
\ar[rr]^{\theta _{0}} &&  V_{0,1}^{(k)}} ,\\ 
(x,(a_{i,j},1\leq i \leq k-1,1\leq j \leq \rg)) \longmapsto 
(x,\sum_{j=1}^{\rg} a_{1,j}\varepsilon _{j}(x),\ldots,
\sum_{j=1}^{\rg} a_{k-1,j}\varepsilon _{j}(x)) \end{array} $$
is a bijective morphism. The open subset ${\goth b}_{\r}$ has a cover by open subsets 
$V$ such that for some $\poi e1{,\ldots,}{n}{}{}{}$ in ${\goth b}$, 
$\poi x{}{,\ldots,}{}{\varepsilon }{1}{\rg},\poi e1{,\ldots,}{n}{}{}{}$ is a basis of 
${\goth b}$ for all $x$ in $V$. Then there exist regular functions 
$\poi {\varphi }1{,\ldots,}{\rg}{}{}{}$ on $V\times {\goth b}$ such that 
$$ v-\sum_{j=1}^{\rg} \varphi _{j}(x,v)\varepsilon _{j}(x) \in 
{\mathrm {span}}(\poi e1{,\ldots,}{n}{}{}{})$$
for all $(x,v)$ in $V\times {\goth b}$, so that the restriction of $\theta _{0}$ to 
$V\times {\mathrm {M}_{k-1,\rg}}(\k)$ is an isomorphism onto 
${\goth X}_{0,k}\cap V\times {\goth b}^{k-1}$ whose inverse is 
$$(\poi x1{,\ldots,}{k}{}{}{}) \longmapsto 
(x_{1},((\poi {x_{1},x_{i}}{}{,\ldots,}{}{\varphi }{1}{\rg}),i=2,\ldots,k)) .$$
As a result, $\theta _{0}$ is an isomorphism and $V_{0,1}^{(k)}$ is a smooth variety. 
Since ${\goth b}_{\r}$ is a smooth open subset of the vector space ${\goth b}$, there 
exists a regular differential form $\omega $ of top degree on 
${\goth b}_{\r}\times {\mathrm {M}_{k-1,\rg}}(\k)$, without zero. Then 
${\theta _{0}}_{*}(\omega )$ is a regular differential form of top degree on 
$V_{0,1}^{(k)}$, without zero.

(ii) By definition, $V_{1}^{(k)}$ is the intersection of ${\cal C}^{(k)}$ and the open 
subset ${\goth g}_{\r}\times {\goth g}^{k-1}$ of ${\goth g}^{k}$. Hence $V_{1}^{(k)}$ is 
an open subset of ${\cal C}^{(k)}$. For $x_{1}$ in ${\goth g}_{\r}$, 
$(\poi x1{,\ldots,}{k}{}{}{})$ is in $V_{1}^{(k)}$ if and only if 
$\poi x2{,\ldots,}{k}{}{}{}$ are in ${\goth g}^{x_{1}}$ by Lemma~\ref{lisc1},(ii) 
since ${\goth g}^{x_{1}}$ is in $G.X$. According to~\cite[Theorem 9]{Ko}, for $x$ in 
${\goth g}_{\r}$, $\poi x{}{,\ldots,}{}{\varepsilon }{1}{\rg}$ is a basis of 
${\goth g}^{x}$. Hence the map 
$$ \begin{array}{ccc} \xymatrix{{\goth g}_{\r}\times {\mathrm {M}_{k-1,\rg}}(\k) 
\ar[rr]^{\theta } && V_{1}^{(k)}} ,\\ 
(x,(a_{i,j},1\leq i \leq k-1,1\leq j \leq \rg)) \longmapsto 
(x,\sum_{j=1}^{\rg} a_{1,j}\varepsilon _{j}(x),\ldots,
\sum_{j=1}^{\rg} a_{k-1,j}\varepsilon _{j}(x)) \end{array} $$
is a bijective morphism. The open subset ${\goth g}_{\r}$ has a cover by open subsets 
$V$ such that for some $\poi e1{,\ldots,}{2n}{}{}{}$ in ${\goth g}$, 
$\poi x{}{,\ldots,}{}{\varepsilon }{1}{\rg},\poi e1{,\ldots,}{2n}{}{}{}$ is a basis of 
${\goth g}$ for all $x$ in $V$. Then there exist regular functions 
$\poi {\varphi }1{,\ldots,}{\rg}{}{}{}$ on $V\times {\goth g}$ such that 
$$ v-\sum_{j=1}^{\rg} \varphi _{j}(x,v)\varepsilon _{j}(x) \in 
{\mathrm {span}}(\poi e1{,\ldots,}{2n}{}{}{})$$
for all $(x,v)$ in $V\times {\goth g}$, so that the restriction of $\theta $ to 
$V\times {\mathrm {M}_{k-1,\rg}}(\k)$ is an isomorphism onto 
${\cal C}^{(k)}\cap V\times {\goth b}^{k-1}$ whose inverse is 
$$(\poi x1{,\ldots,}{k}{}{}{}) \longmapsto 
(x_{1},((\poi {x_{1},x_{i}}{}{,\ldots,}{}{\varphi }{1}{\rg}),i=2,\ldots,k)) .$$
As a result, $\theta $ is an isomorphism and $V_{1}^{(k)}$ is a smooth variety. Since
${\goth g}_{\r}$ is a smooth open subset of the vector space ${\goth g}$, there exists
a regular differential form $\omega $ of top degree on 
${\goth g}_{\r}\times {\mathrm {M}_{k-1,\rg}}(\k)$, without zero. Then 
$\theta _{*}(\omega )$ is a regular differential form of top degree on $V_{1}^{(k)}$, without zero.

(iii) Since ${\cal X}_{\r}$ is the inverse image of ${\goth g}_{\r}$ by the canonical
projection $\xymatrix{{\cal X}\ar[r] & {\goth g}}$, $W_{1}^{(k)}$ is the intersection of 
${\cal C}_{\x}^{(k)}$ and ${\cal X}_{\r}\times {\cal X}^{k-1}$. Hence $W_{1}^{(k)}$ is an
open subset of ${\cal C}_{\x}^{(k)}$. Moreover, $W_{1}^{(k)}$ is the image of 
$\sqx G{V_{0,1}^{(k)}}$ by $\gamma _{\x}$. Since the maps 
$\poi {\varepsilon }1{,\ldots,}{\rg}{}{}{}$ are $G$-equivariant, the map
\begin{eqnarray*}
\xymatrix{ G\times {\goth b}_{\r}\times {\mathrm {M}_{k-1,\rg}}(\k) \ar[rr] && 
W_{1}^{(k)}}, \\
(g,x,a_{i,j},1\leq i\leq k-1,1\leq j \leq \rg) \longmapsto 
\gamma _{\x}(\overline{g,\theta _{0}(x,a_{i,j},1\leq i\leq k-1,1\leq j \leq \rg)})
\end{eqnarray*}
defines through the quotient a surjective morphism 
$$ \xymatrix{\sqx G{{\goth b}_{\r}}\times 
{\mathrm {M}_{k-1,\rg}}(\k) \ar[rr]^{\theta _{\x}} && W_{1}^{(k)}} .$$
Let $\varpi _{2}$ be the canonical projection 
$$ \xymatrix{ {\goth g}_{\r}\times {\mathrm {M}_{k-1,\rg}}(\k) \ar[rr] && 
{\mathrm {M}_{k-1,\rg}}(\k)} .$$
According to Lemma~\ref{lxv},(iii), the restriction of $\chi _{\n}$ to 
$\sqx G{{\goth b}_{\r}}$ is an isomorphism onto ${\cal X}_{\r}$. So denote by 
$\varpi _{1}$ the morphism 
$$ \xymatrix{ W_{1}^{(k)} \ar[rr]^{\varpi _{1}} && \sqx G{{\goth b}_{\r}}}, \qquad
(\poi x1{,\ldots,}{k}{}{}{},\poi y1{,\ldots,}{k}{}{}{}) \longmapsto 
\chi _{\n}^{-1}(x_{1},y_{1}) .$$
Then $\theta _{\x}$ is an isomorphism whose inverse is given by 
$$ x \longmapsto (\varpi _{1}(x),\varpi _{2}\rond \theta ^{-1} \rond \eta (x))$$
since $\eta (W_{1}^{(k)}) = V_{1}^{(k)}$. In particular, $W_{1}^{(k)}$ is a smooth open 
subset of ${\cal C}_{\x}^{(k)}$. According to Proposition~\ref{pxv},(ii), there exists
a regular differential form $\omega $ of top degree on 
$\sqx G{{\goth b}_{\r}}\times M_{k-1,\rg}(k)$, without zero. Then 
${\theta _{\x}}_{*}(\omega )$ is a regular differential form of top degree on 
$W_{1}^{(k)}$, without zero.
\end{proof}

\begin{coro}\label{cisc2}
Let $j=1,\ldots,k$.

{\rm (i)} The morphism $\tau _{0,k}$ is birational. More precisely, the restriction of
$\tau _{0,k}$ to $\tau _{0,k}^{-1}(V_{0,j}^{(k)})$ is an isomorphism onto $V_{0,j}^{(k)}$.

{\rm (ii)} The morphism $\tau _{k}$ is birational. More precisely, the restriction of
$\tau _{k}$ to $\tau _{k}^{-1}(V_{j}^{(k)})$ is an isomorphism onto $V_{j}^{(k)}$.

{\rm (iii)} The morphism $\tau _{*,k}$ is birational. More precisely, the restriction of
$\tau _{*,k}$ to $\tau _{*,k}^{-1}(W_{j}^{(k)})$ is an isomorphism onto $W_{j}^{(k)}$.
\end{coro}

\begin{proof}
According to the above remarks, we can suppose $j=1$.

(i) For $(\poi x1{,\ldots,}{k}{}{}{})$ in $V_{0,1}^{(k)}$ and for $u$ in $X$ containing
$\poi x1{,\ldots,}{k}{}{}{}$, $u={\goth g}^{x_{1}}$ since $x_{1}$ is regular. Hence   
the restriction of $\tau _{0,k}$ to $\tau _{0,k}^{-1}(V_{0,1}^{(k)})$ is injective.
According to Lemma~\ref{lisc1},(i), $V_{0,1}^{(k)}$ is the image of 
$\tau _{0,k}^{-1}(V_{0,1}^{(k)})$ by $\tau _{0,k}$. So, by Lemma~\ref{lisc2},(i) and 
Zariski's Main Theorem ~\cite[\S 9]{Mu}, the restriction of $\tau _{0,k}$ to 
$\tau _{0,k}^{-1}(V_{0,1}^{(k)})$ is an isomorphism onto $V_{0,1}^{(k)}$.

(ii) For $(\poi x1{,\ldots,}{k}{}{}{})$ in $V_{1}^{(k)}$ and for $u$ in $G.X$ containing
$\poi x1{,\ldots,}{k}{}{}{}$, $u={\goth g}^{x_{1}}$ since $x_{1}$ is regular. Hence   
the restriction of $\tau _{k}$ to $\tau _{k}^{-1}(V_{1}^{(k)})$ is injective.
According to Lemma~\ref{lisc1},(ii), $V_{1}^{(k)}$ is the image of 
$\tau _{k}^{-1}(V_{1}^{(k)})$ by $\tau _{k}$. So, by Lemma~\ref{lisc2},(ii) and 
Zariski's Main Theorem ~\cite[\S 9]{Mu}, the restriction of $\tau _{k}$ to 
$\tau _{k}^{-1}(V_{1}^{(k)})$ is an isomorphism onto $V_{1}^{(k)}$.

(iii) The variety $\sqx G{{\cal E}_{0}^{(k)}}$ identifies to a closed subvariety of 
$G/B\times {\cal E}^{(k)}$. For $(\poi x1{,\ldots,}{k}{}{}{},\poi y1{,\ldots,}{k}{}{}{})$
in $W_{1}^{(k)}$ and $(u,v)$ in $G/B\times G.X$ such that 
$(u,v,\poi x1{,\ldots,}{k}{}{}{})$ is in $\sqx G{{\cal E}_{0}^{(k)}}$, $(x_{1},y_{1})$ is
in ${\cal X}_{\r}$, $\chi _{\n}(u,x_{1})=(x_{1},y_{1})$ and $v={\goth g}^{x_{1}}$ since 
$x_{1}$ is regular. Moreover, $u$ is unique by 
Lemma~\ref{lxv},(iii). Hence the restriction of $\tau _{*,k}$ to 
$\tau _{*,k}^{-1}(W_{1}^{(k)})$ is injective. According to Lemma~\ref{lisc1},(iii), 
$W_{1}^{(k)}$ is the image of $\tau _{*,k}^{-1}(W_{1}^{(k)})$ by $\tau _{*,k}$. So, by 
Lemma~\ref{lisc2},(iii) and Zariski's Main Theorem ~\cite[\S 9]{Mu}, the restriction of 
$\tau _{*,k}$ to $\tau _{*,k}^{-1}(W_{1}^{(k)})$ is an isomorphism onto $W_{1}^{(k)}$. 
\end{proof}

Set:
$$ V_{0}^{(k)} := V_{0,1}^{(k)} \cup V_{0,2}^{(k)}, \qquad 
V^{(k)} := V_{1}^{(k)} \cup V_{2}^{(k)}, \qquad
W^{(k)} := W_{1}^{(k)} \cup W_{2}^{(k)} .$$

\begin{lemma}\label{l2isc2}
{\rm (i)} The set $\tau _{0,k}^{-1}(V_{0}^{(k)})$ is a big open subset of 
${\cal E}_{0}^{(k)}$.

{\rm (ii)} The set $\tau _{k}^{-1}(V^{(k)})$ is a big open subset of 
${\cal E}^{(k)}$.

{\rm (i)} The set $\tau _{*,k}^{-1}(W^{(k)})$ is a big open subset of 
$\sqx G{{\cal E}_{0}^{(k)}}$.
\end{lemma}

\begin{proof}
(i) Let $\Sigma $ be an irreducible component of 
${\cal E}_{0}^{(k)}\setminus \tau _{0,k}^{-1}(V_{0}^{(k)})$. Since $V_{0}^{(k)}$ is a
$B$-invariant open cone, $\Sigma $ is a $B$-invariant closed subset of 
${\cal E}_{0}^{(k)}$ such that $\Sigma \cap \pi _{0,k}^{-1}(u)$ is a closed cone of 
$\pi _{0,k}^{-1}(u)$ for all $u$ in $\pi _{0,k}(\Sigma )$. As a result 
$\pi _{0,k}(\Sigma )$ is a closed subset of $X$. Indeed, 
$\pi _{0,k}(\Sigma )\times \{0\}=\Sigma \cap X\times \{0\}$. 
For $u$ in $\pi _{0,k}(\Sigma )$, denote by $\Sigma _{u}$ the closed subvariety of 
$u^{k}$ such that $\pi _{0,k}^{-1}(u)\cap \Sigma = \{u\}\times \Sigma _{u}$.

Suppose that $\Sigma $ has codimension $1$ in ${\cal E}_{0}^{(k)}$. A contradiction is 
expected. Then $\pi _{0,k}(\Sigma )$ has codimension at most $1$ in $X$. Since $X'$ is a 
big open subset of $X$, for all $u$ in a dense open subset of $\pi _{0,k}(\Sigma )$, 
$u\cap {\goth b}_{\r}$ is not empty. If $\pi _{0,k}(\Sigma )$ has 
codimension $1$ in $X$, then $\Sigma _{u} = u^{k}$ for all $u$ in $\pi _{0,k}(\Sigma )$. 
Hence $\pi _{0,k}(\Sigma )=X$ and for all $u$ in a dense open subset 
of $X'$, $\Sigma _{u}$ has dimension $k\rg -1$. For such $u$, the image of $\Sigma _{u}$
by the first projection onto $u$ is not dense in $u$ since $u\cap {\goth b}_{\r}$ is 
not empty. Hence the image of $\Sigma _{u}$ by the second projection is equal to $u$ 
since $\Sigma _{u}$ has codimension $1$ in $u^{k}$. It is impossible since this image 
is contained in $u\setminus {\goth b}_{\r}$.

(ii) Let $\Sigma $ be an irreducible component of 
${\cal E}^{(k)}\setminus \tau _{k}^{-1}(V^{(k)})$. Since $V^{(k)}$ is a
$G$-invariant open cone, $\Sigma $ is a $G$-invariant closed subset of ${\cal E}^{(k)}$ 
such that $\Sigma \cap \pi _{k}^{-1}(u)$ is a closed cone of $\pi _{k}^{-1}(u)$ for all 
$u$ in $\pi _{k}(\Sigma )$. As a result $\pi _{k}(\Sigma )$ is a closed subset of $G.X$. 
Indeed, $\pi _{k}(\Sigma )\times \{0\}=\Sigma \cap G.X\times \{0\}$. For $u$ in 
$\pi _{k}(\Sigma )$, denote by $\Sigma _{u}$ the closed subvariety of $u^{k}$ such that 
$\pi _{k}^{-1}(u)\cap \Sigma = \{u\}\times \Sigma _{u}$.

Suppose that $\Sigma $ has codimension $1$ in ${\cal E}^{(k)}$. A contradiction is 
expected. Then $\pi _{k}(\Sigma )$ has codimension at most $1$ in $X$. Since $G.X'$ is a 
big open subset of $G.X$, for all $u$ in a dense open subset of $\pi _{k}(\Sigma )$, 
$u\cap {\goth g}_{\r}$ is not empty. If $\pi _{k}(\Sigma )$ has codimension $1$ in $G.X$,
then $\Sigma _{u} = u^{k}$ for all $u$ in $\pi _{k}(\Sigma )$. Hence 
$\pi _{k}(\Sigma )=G.X$ and for all $u$ in a dense open subset of $G.X'$, $\Sigma _{u}$ 
has dimension $k\rg -1$. For such $u$, the image of $\Sigma _{u}$
by the first projection onto $u$ is not dense in $u$ since $u\cap {\goth g}_{\r}$ is 
not empty. Hence the image of $\Sigma _{u}$ by the second projection is equal to $u$ 
since $\Sigma _{u}$ has codimension $1$ in $u^{k}$. It is impossible since this image 
is contained in $u\setminus {\goth g}_{\r}$.

(iii) Let $\Sigma $ be an irreducible component of 
$\sqx G{{\cal E}_{0}^{(k)}}\setminus \tau _{*,k}^{-1}(W^{(k)})$. Since $W^{(k)}$ is a
$G$-invariant open cone, $\Sigma $ is a $G$-invariant closed subset of 
$\sqx G{{\cal E}_{0}^{(k)}}$. So, for some $B$-invariant closed subset $\Sigma _{0}$ of
${\cal E}_{0}^{(k)}$, $\Sigma = \sqx G{\Sigma _{0}}$. Moreover, $\Sigma _{0}$ is contained
in ${\cal E}_{0,k}\setminus \tau _{0,k}^{-1}(V_{0}^{(k)})$. According to (i), 
$\Sigma _{0}$ has codimension at least $2$ in ${\cal E}_{0}^{(k)}$. Hence $\Sigma $ has 
codimension at least $2$ in $\sqx G{{\cal E}_{0}^{(k)}}$.
\end{proof}

\subsection{} \label{isc3}
For $2\leq k'\leq k$, the projection
$$ \xymatrix{ {\goth g}^{k} \ar[rr] && {\goth g}^{k'}}, \qquad
(\poi x1{,\ldots,}{k}{}{}{}) \longmapsto (\poi x1{,\ldots,}{k'}{}{}{})$$
induces the projections 
$$ \xymatrix{ {\goth X}_{0,k} \ar[rr] && {\goth X}_{0,k'}}, \qquad
\xymatrix{ V_{0,j}^{(k)} \ar[rr] && V_{0,j}^{(k')}} .$$
Set:
$$ V_{0,1,2}^{(k)} := V_{0,1}^{(k)} \cap V_{0,2}^{(k)} .$$

\begin{lemma}\label{lisc3}
Let $\omega $ be a regular differential form of top degree on $V_{0,1}^{(k)}$, without 
zero. Denote by $\omega '$ its restriction to $V_{0,1,2}^{(k)}$.

{\rm (i)} For $\varphi $ in $\k[V_{0,1}^{(k)}]$, if $\varphi $ has no zero then $\varphi $
is in $\k^{*}$.

{\rm (ii)} For some invertible element $\psi $ of $\k[V_{0,1,2}^{(2)}]$, 
$\omega ' = \psi {\sigma _{2}}_{*}(\omega ')$.

{\rm (iii)} The function $\psi (\psi \rond \sigma _{2})$ on $V_{0,1,2}^{(k)}$ is equal 
to $1$. 
\end{lemma}

\begin{proof}
The existence of $\omega $ results from Lemma~\ref{lisc2},(i).

(i) According to Lemma~\ref{lisc2},(i), there is an isomorphism $\theta _{0}$ from 
${\goth b}_{\r}\times {\mathrm {M}_{k-1,\rg}}(\k)$ onto $V_{0,1}^{(k)}$. Since 
$\varphi $ is invertible, $\varphi \rond \theta _{0}$ is an invertible element of 
$\k[{\goth b}_{\r}]$. According to Lemma~\ref{lxv},(i), 
$\k[{\goth b}_{\r}]=\k[{\goth b}]$. Hence $\varphi $ is in $\k^{*}$.

(ii) The open subset $V_{0,1,2}^{(k)}$ is invariant under $\sigma _{2}$ so that 
$\omega '$ and ${\sigma _{2}}_{*}(\omega ')$ are regular differential forms of top
degree on $V_{0,1,2}^{(k)}$, without zero. Then for some invertible element $\psi $ 
of $\k[V_{0,1,2}^{(k)}]$, $\omega ' = \psi {\sigma _{2}}_{*}(\omega ')$. Let 
$O_{2}$ be the set of elements $(x,a_{i,j},1\leq i \leq k-1,1\leq j\leq \rg)$ of 
${\goth b}_{\r}\times {\mathrm {M}_{k-1,\rg}}(\k)$ such that 
$$ a_{1,1}\varepsilon _{1}(x) + \cdots + a_{1,\rg} \varepsilon _{\rg}(x) \in 
{\goth b}_{\r}.$$
Then $O_{2}$ is the inverse image of $V_{0,1,2}^{(k)}$ by $\theta _{0}$. As a result,
$\k[V_{0,1,2}^{(k)}]$ is a polynomial algebra over $\k[V_{0,1,2}^{(2)}]$ since for 
$k=2$, $O_{2}$ is the inverse image by $\theta _{0}$ of $V_{0,1,2}^{(2)}$. Hence 
$\psi $ is in $\k[V_{0,1,2}^{(2)}]$ since $\psi $ is invertible.

(iii) Since the restriction of $\sigma _{2}$ to $V_{0,1,2}^{(k)}$ is an involution,
$$ {\sigma _{2}}_{*}(\omega ') = (\psi \rond \sigma _{2}) \omega ' = 
(\psi \rond \sigma _{2})\psi {\sigma _{2}}_{*}(\omega '),$$
whence $(\psi \rond \sigma _{2})\psi = 1$.
\end{proof}

\begin{coro}\label{cisc3}
The function $\psi $ is invariant under the action of $B$ in $V_{0,1,2}^{(k)}$ and for
some sequence $m_{\alpha },\alpha \in {\cal R}_{+}$ in ${\Bbb Z}$, 
$$ \psi (\poi x1{,\ldots,}{k}{}{}{}) = \pm \prod_{\alpha \in {\cal R}_{+}}
(\alpha (x_{1})\alpha (x_{2})^{-1})^{m_{\alpha }},$$
for all $(\poi x1{,\ldots,}{k}{}{}{})$ in ${\goth h}_{\r}^{2}\times {\goth h}^{k-2}$.
\end{coro}

\begin{proof}
First of all, since $V_{0,1}^{(k)}$ and $V_{0,2}^{(k)}$ are invariant under the action of
$B$ in ${\goth X}_{0,k}$, so is $V_{0,1,2}^{(k)}$. Let $g$ be in $B$. Since $\omega $ has 
no zero, $g.\omega =p_{g}\omega $ for some invertible element $p_{g}$ of 
$\k[V_{0,1}^{(k)}]$. By Lemma~\ref{lisc3},(i), $p_{g}$ is in $\k^{*}$. Since 
$\sigma _{2}$ is a $B$-equivariant isomorphism from $V_{0,1}^{(k)}$ onto $V_{0,2}^{(k)}$, 
$$ g.{\sigma _{2}}_{*}(\omega )=p_{g}{\sigma _{2}}_{*}(\omega ) \quad  \text{and} \quad
p_{g} \omega ' = g.\omega ' = (g.\psi ) g.{\sigma _{2}}_{*}(\omega ') = 
p_{g} (g.\psi ) {\sigma _{2}}_{*}(\omega '),$$ 
whence $g.\psi =\psi $.

The open subset ${\goth h}_{\r}^{2}$ of ${\goth h}^{2}$ is the complement of the 
nullvariety of the function
$$ (x,y) \longmapsto \prod_{\alpha \in {\cal R}_{+}} \alpha (x)\alpha (y).$$
Then, by Lemma~\ref{lisc3},(ii), for some $a$ in $\k^{*}$ and for some 
sequences $m_{\alpha },\alpha \in {\cal R}_{+}$ and 
$n_{\alpha },\alpha \in {\cal R}_{+}$ in ${\Bbb Z}$,
$$ \psi (\poi x1{,\ldots,}{k}{}{}{}) = a \prod_{\alpha \in {\cal R}_{+}}
\alpha (x_{1})^{m_{\alpha }}\alpha (x_{2})^{n_{\alpha }},$$
for all $(\poi x1{,\ldots,}{k}{}{}{})$ in ${\goth h}_{\r}^{2}\times {\goth h}^{k-2}$.
Then, by Lemma~\ref{lisc3},(iii),
$$ a^{2} \prod_{\alpha \in {\cal R}_{+}} \alpha (x)^{m_{\alpha }+n_{\alpha }}
\alpha (y)^{m_{\alpha }+n_{\alpha }} = 1,$$ 
for all $(x,y)$ in ${\goth h}_{\r}^{2}$. Hence $a^{2}=1$ and $m_{\alpha }+n_{\alpha }=0$ 
for all $\alpha $ in ${\cal R}_{+}$.
\end{proof}

For $\alpha $ a positive root, denote by ${\goth h}_{\alpha }$ the kernel of $\alpha $
and set:
$$ V_{\alpha } := {\goth h}_{\alpha } \oplus {\goth g}^{\alpha } .$$
Denote by $\thetaup _{\alpha }$ the map
$$ \xymatrix{ \k \ar[rr]^{\thetaup _{\alpha }} && X}, \qquad 
t \longmapsto \exp(t\ad x_{\alpha }).{\goth h} .$$
According to~\cite[Ch. VI, Theorem 1]{Sh}, $\thetaup _{\alpha }$ has a regular extension 
to ${\Bbb P}^{1}(\k)$. Set $Z_{\alpha } := \thetaup _{\alpha }({\Bbb P}^{1}(\k))$.
Denote again by $\alpha $ the element of ${\goth b}^{*}$ extending $\alpha $ and equal to
$0$ on ${\goth u}$.

\begin{lemma}\label{l2isc3}
Let $\alpha $ be in ${\cal R}_{+}$ and let $x_{0}$ and $y_{0}$ be subregular in 
${\goth h}_{\alpha }$. Set:
$$ E := \k x_{0} \oplus \k H_{\alpha } \oplus {\goth g}^{\alpha }, \quad
E_{*} := x_{0} \oplus \k H_{\alpha } \oplus {\goth g}^{\alpha } , \quad
E_{*,1} := x_{0} \oplus \k H_{\alpha } \oplus ({\goth g}^{\alpha }\setminus \{0\}),
\quad 
E_{*,2} = y_{0} \oplus \k H_{\alpha } \oplus ({\goth g}^{\alpha }\setminus \{0\}).$$

{\rm (i)} For $x$ in $E_{*}$, the centralizer of $x$ in ${\goth b}$ is contained in 
${\goth h}_{\alpha }+E$.

{\rm (ii)} For $V$ subspace of dimension $\rg$ of ${\goth h}_{\alpha }+E$, $V$ is in 
$X$ if and only if it is in $Z_{\alpha }$.

{\rm (iii)} The intersection of $E_{*,1}\times E_{*,2}$ and ${\goth X}_{0,2}$ is the 
nullvariety of the function 
$$ (x,y) \longmapsto \dv {x_{-\alpha }}y \alpha (x) - \dv {x_{-\alpha }}x \alpha (y) $$
on $E_{*,1}\times E_{*,2}$.
\end{lemma}

\begin{proof}
(i) If $x$ is regular semisimple, its component on $H_{\alpha }$ is different from $0$
so that ${\goth g}^{x}=\thetaup _{\alpha }(t)$ for some $t$ in $\k$. Suppose that 
$x$ is not regular semisimple. Then $x$ is in $x_{0}+{\goth g}^{\alpha }$. Hence
${\goth g}^{x}\cap {\goth b}$ is contained in ${\goth h}_{\alpha }+E$ since so is 
${\goth g}^{x_{0}}\cap {\goth b}$.

(ii) All element of $Z_{\alpha }$ is contained in ${\goth h}_{\alpha }+E$. Let $V$ be an
element of $X$, contained in ${\goth h}_{\alpha }+E$. According 
to~\cite[Corollary 4.3]{CZ}, $V$ is an algebraic commutative subalgebra of dimension 
$\rg$ of ${\goth b}$. By (i), $V=\thetaup _{\alpha }(t)$ for some $t$ in $\k$ if $V$ is a
Cartan subalgebra. Otherwise, $x_{\alpha }$ is in $V$. Then 
$V=\thetaup _{\alpha }(\infty )$ since $\thetaup _{\alpha }(\infty )$ is the centralizer 
of $x_{\alpha }$ in ${\goth h}_{\alpha }+E$.

(iii) Let $(x,y)$ be in $E_{*,1}\times E_{*,2}\cap {\goth X}_{0,2}$. According to 
Lemma~\ref{lisc1},(i), for some $V$ in $X$, $x$ and $y$ are in $V$. By (i) and (ii), 
$V=\thetaup _{\alpha }(t)$ for some $t$ in ${\Bbb P}^{1}(\k)$. For $t$ in $\k$, 
$$ x = x_{0}+s(H_{\alpha }-2t x_{\alpha }) \quad  \text{and} \quad 
y = y_{0} + s'(H_{\alpha } - 2t x_{\alpha })$$
for some $s$, $s'$ in $\k$, whence the equality of the assertion. For 
$t=\infty $, $$ x = x_{0}+sx_{\alpha } \quad  \text{and} \quad 
y = y_{0} + s'x_{\alpha } $$
for some $s$, $s'$ in $\k$ so that $\alpha (x)=\alpha (y) = 0$. Conversely,
let $(x,y)$ be in $E_{*,1}\times E_{*,2}$ such that
$$ \dv {x_{-\alpha }}y \alpha (x) - \dv {x_{-\alpha }}x \alpha (y) = 0 .$$
If $\alpha (x)=0$ then $\alpha (y) = 0$ and $x$ and $y$ are in 
$V_{\alpha }=\thetaup _{\alpha }(\infty )$. If $\alpha (x)\neq 0$, then 
$\alpha (y) \neq 0$ and
$$x \in \thetaup _{\alpha } (-\frac{\dv {x_{-\alpha }}x}{\alpha (x)}) 
\quad  \text{and} \quad
y \in \thetaup _{\alpha } (-\frac{\dv {x_{-\alpha }}x}{\alpha (x)}) ,$$
whence the assertion.
\end{proof}

\begin{prop}\label{pisc3}
There exists on $V_{0}^{(k)}$ a regular differential form of top degree without zero.
\end{prop}

\begin{proof}
According to Corollary~\ref{cisc3}, it suffices to prove $m_{\alpha }=0$ for all 
$\alpha $ in ${\cal R}_{+}$. Indeed, if so, by Corollary~\ref{cisc3}, $\psi =\pm 1$ on 
the open subset $B.({\goth h}_{\r}^{2}\times {\goth h}^{k-2})$ of $V_{0}^{(k)}$ so 
that $\psi =\pm 1$ on $V_{0,1,2}^{(k)}$. Then, by Lemma~\ref{lisc3},(ii), $\omega $ and 
$\pm {\sigma _{2}}_{*}(\omega )$ have the same restriction to $V_{0,1,2}^{(k)}$ so that
there exists a regular differential form of top degree $\tilde{\omega }$ on $V_{0}^{(k)}$
whose restrictions to $V_{0,1}^{(k)}$ and $V_{0,2}^{(k)}$ are $\omega $ and 
$\pm {\sigma _{2}}_{*}(\omega )$ respectively. Moreover, $\tilde{\omega }$ has no zero
since so has $\omega $.

Since $\psi $ is in $\k[V_{0,1,2}^{(2)}]$ by Lemma~\ref{lisc3},(ii), we can suppose $k=2$.
Let $\alpha $ be in ${\cal R}_{+}$, $E$, $E_{*}$,$E_{*,1}$, $E_{*,2}$ as in 
Lemma~\ref{l2isc3}. Suppose $m_{\alpha }\neq 0$. A contradiction is expected. According 
to Lemma~\ref{l2isc3},(iii), the restriction of $\psi $ to 
$E_{*,1}\times E_{*,2}\cap V_{0,1,2}^{(2)}$ is given by 
$$ \psi (x,y) = a \dv {x_{-\alpha }}x^{m}\dv {x_{-\alpha }}y^{n} ,$$  
with $a$ in $\k^{*}$ and $(m,n)$ in ${\Bbb Z}^{2}$ since $\psi $ is an invertible element
of $\k[V_{0,1,2}^{(2)}]$. According to Lemma~\ref{lisc3},(iii), $n=-m$ and $a=\pm 1$.
Interchanging the role of $x$ and $y$, we can suppose $m$ in ${\Bbb N}$. For $(x,y)$ in 
$E_{*,1}\times E_{*,2}\cap V_{0,1,2}^{(2)}$ such that $\alpha (x)\neq 0$, 
$\alpha (y)\neq 0$ and 
$$ \psi (x,y) = \pm \dv {x_{-\alpha }}x^{m} 
(\frac{\dv {x_{-\alpha }}x \alpha (y)}{\alpha (x)})^{-m} = 
\pm \alpha (x)^{m} \alpha (y)^{-m} .$$ 
As a result, by Corollary~\ref{cisc3}, for $x$ in $x_{0}+\k^{*} H_{\alpha }$ and 
$y$ in $y_{0}+\k^{*}H_{\alpha }$, 
\begin{eqnarray}\label{eqisc3}
\pm \alpha (x)^{m}\alpha (y)^{-m} = \pm \prod_{\gamma \in {\cal R}_{+}}
\gamma (x)^{m_{\gamma }}\gamma (y)^{-m_{\gamma }} .
\end{eqnarray}
For $\gamma $ in ${\cal R}_{+}$, 
$$ \gamma (x) = \gamma (x_{0}) + \frac{1}{2}\alpha (x)\gamma (H_{\alpha }) 
\quad  \text{and} \quad
\gamma (y) = \gamma (y_{0}) + \frac{1}{2}\alpha (y)\gamma (H_{\alpha }) .$$
Since $m$ is in ${\Bbb N}$,  
\begin{eqnarray}\label{eq2isc3}
\pm \alpha (x)^{m} \prod_{\mycom{\gamma \in {\cal R}_{+}}{m_{\gamma }>0}}
(\gamma (y_{0}) + \frac{1}{2}\alpha (y)\gamma (H_{\alpha }))^{m_{\gamma }}
\prod_{\mycom{\gamma \in {\cal R}_{+}}{m_{\gamma }< 0}}
(\gamma (x_{0}) + \frac{1}{2}\alpha (x)\gamma (H_{\alpha }))^{-m_{\gamma }} =  \\
\nonumber \pm \alpha (y)^{m} \prod_{\mycom{\gamma \in {\cal R}_{+}}{m_{\gamma }>0}}
(\gamma (x_{0}) + \frac{1}{2}\alpha (x)\gamma (H_{\alpha }))^{m_{\gamma }}
\prod_{\mycom{\gamma \in {\cal R}_{+}}{m_{\gamma } < 0}}
(\gamma (y_{0}) + \frac{1}{2}\alpha (y)\gamma (H_{\alpha }))^{-m_{\gamma }} .
\end{eqnarray}
For $m_{\alpha }$ positive, the terms of lowest degree in $(\alpha (x),\alpha (y))$ of 
left and right sides are
$$ \pm \alpha (x)^{m} \alpha (y)^{m_{\alpha }}
\prod_{\mycom{\gamma \in {\cal R}_{+}\setminus \{\alpha \}}{m_{\gamma }>0}} 
\gamma (y_{0})^{m_{\gamma }} 
\prod_{\mycom{\gamma \in {\cal R}_{+}\setminus \{\alpha \}}{m_{\gamma }<0}} 
\gamma (x_{0})^{-m_{\gamma }} 
\quad  \text{and} \quad 
\pm \alpha (y)^{m} \alpha (x)^{m_{\alpha }} 
\prod_{\mycom{\gamma \in {\cal R}_{+}\setminus \{\alpha \}}{m_{\gamma }>0}} 
\gamma (x_{0})^{m_{\gamma }} 
\prod_{\mycom{\gamma \in {\cal R}_{+}\setminus \{\alpha \}}{m_{\gamma }<0}} 
\gamma (y_{0})^{-m_{\gamma }} $$
respectively and for $m_{\alpha }$ negative, the terms of lowest degree in 
$(\alpha (x),\alpha (y))$ of left and right sides are
$$ \pm \alpha (x)^{m+m_{\alpha }} 
\prod_{\mycom{\gamma \in {\cal R}_{+}\setminus \{\alpha \}}{m_{\gamma }>0}} 
\gamma (y_{0})^{m_{\gamma }} 
\prod_{\mycom{\gamma \in {\cal R}_{+}\setminus \{\alpha \}}{m_{\gamma }<0}} 
\gamma (x_{0})^{-m_{\gamma }} 
\quad  \text{and} \quad 
\pm \alpha (y)^{m+m_{\alpha }} 
\prod_{\mycom{\gamma \in {\cal R}_{+}\setminus \{\alpha \}}{m_{\gamma }>0}} 
\gamma (x_{0})^{m_{\gamma }} 
\prod_{\mycom{\gamma \in {\cal R}_{+}\setminus \{\alpha \}}{m_{\gamma }<0}} 
\gamma (y_{0})^{-m_{\gamma }} $$
respectively. From the equality of these terms, we deduce $m = \pm m_{\alpha }$ and
$$ \prod_{\mycom{\gamma \in {\cal R}_{+}\setminus \{\alpha \}}{m_{\gamma }>0}} 
\gamma (y_{0})^{m_{\gamma }} 
\prod_{\mycom{\gamma \in {\cal R}_{+}\setminus \{\alpha \}}{m_{\gamma }<0}} 
\gamma (x_{0})^{-m_{\gamma }} =
\pm \prod_{\mycom{\gamma \in {\cal R}_{+}\setminus \{\alpha \}}{m_{\gamma }>0}} 
\gamma (x_{0})^{m_{\gamma }} 
\prod_{\mycom{\gamma \in {\cal R}_{+}\setminus \{\alpha \}}{m_{\gamma }<0}} 
\gamma (y_{0})^{-m_{\gamma }} .$$
Since the last equality does not depend on the choice of subregular elements 
$x_{0}$ and $y_{0}$ in ${\goth h}_{\alpha }$, this equality remains true for all
$(x_{0},y_{0})$ in ${\goth h}_{\alpha }\times {\goth h}_{\alpha }$. As a result,
as the degrees in $\alpha (x)$ of the left and right sides of Equality (\ref{eq2isc3})
are the same,
\begin{eqnarray}\label{eq3isc3}
m - \sum_{\mycom{\gamma \in {\cal R}_{+}}{m_{\gamma }< 0\;  \text{and} \; 
\gamma (H_{\alpha })\neq 0}} m_{\gamma } = 
\sum_{\mycom{\gamma \in {\cal R}_{+}}{m_{\gamma } > 0 \;  \text{and} \; 
\gamma (H_{\alpha })\neq 0}} m_{\gamma } .
\end{eqnarray}

Suppose $m=m_{\alpha }$. By Equality (\ref{eqisc3}), 
$$ \prod_{\gamma \in {\cal R}_{+}\setminus \{\alpha \}} 
\gamma (x)^{m_{\gamma }}\gamma (y)^{-m_{\gamma }} = \pm 1.$$
Since this equality does not depend on the choice of the subregular elements $x_{0}$ and 
$y_{0}$ in ${\goth h}_{\alpha }$, it holds for all $(x,y)$ in 
${\goth h}_{\r}\times {\goth h}_{\r}$. Hence $m_{\gamma }=0$ for all $\gamma $ in 
${\cal R}_{+}\setminus \{\alpha \}$ and $m=0$ by Equality (\ref{eq3isc3}). It is 
impossible since $m_{\alpha }\neq 0$. Hence $m=-m_{\alpha }$. Then, by 
Equality (\ref{eqisc3})
$$ \prod_{\gamma \in {\cal R}_{+}\setminus \{\alpha \}} 
\gamma (x)^{m_{\gamma }}\gamma (y)^{-m_{\gamma }} = 
\pm \alpha (x)^{2m}\alpha (y)^{-2m} .$$
Since this equality does not depend on the choice of the subregular elements $x_{0}$ and 
$y_{0}$ in ${\goth h}_{\alpha }$, it holds for all $(x,y)$ in 
${\goth h}_{\r}\times {\goth h}_{\r}$. Then $m=0$, whence the contradiction.
\end{proof}

\subsection{} \label{isc4}
For $2\leq k'\leq k$, the projection
$$ \xymatrix{ {\goth g}^{k} \ar[rr] && {\goth g}^{k'}}, \qquad
(\poi x1{,\ldots,}{k}{}{}{}) \longmapsto (\poi x1{,\ldots,}{k'}{}{}{})$$
induces the projections 
$$ \xymatrix{ {\cal C}^{(k)} \ar[rr] && {\cal C}^{(k')}} , \qquad
\xymatrix{ V_{j}^{(k)} \ar[rr] && V_{j}^{(k')}} .$$
Set:
$$ V_{1,2}^{(k)} := V_{1}^{(k)} \cap V_{2}^{(k)} .$$

\begin{lemma}\label{lisc4}
Let $\omega $ be a regular differential form of top degree on $V_{1}^{(k)}$, without 
zero. Denote by $\omega '$ its restriction to $V_{1,2}^{(k)}$.

{\rm (i)} For $\varphi $ in $\k[V_{1}^{(k)}]$, if $\varphi $ has no zero then $\varphi $
is in $\k^{*}$.

{\rm (ii)} For some invertible element $\psi $ of $\k[V_{1,2}^{(2)}]$, 
$\omega ' = \psi {\sigma _{2}}_{*}(\omega ')$.

{\rm (iii)} The function $\psi (\psi \rond \sigma _{2})$ on $V_{1,2}^{(k)}$ is equal 
to $1$. 
\end{lemma}

\begin{proof}
Following the arguments of the proof of Lemma~\ref{lisc3}, the lemma results from
Lemma~\ref{lisc2},(ii).
\end{proof}

\begin{coro}\label{cisc4}
The function $\psi $ is invariant under the action of $G$ in $V_{1,2}^{(k)}$ and for
some sequence $m_{\alpha },\alpha \in {\cal R}_{+}$ in ${\Bbb Z}$, 
$$ \psi (\poi x1{,\ldots,}{k}{}{}{}) = \pm \prod_{\alpha \in {\cal R}_{+}}
(\alpha (x_{1})\alpha (x_{2})^{-1})^{m_{\alpha }},$$
for all $(\poi x1{,\ldots,}{k}{}{}{})$ in ${\goth h}_{\r}^{2}\times {\goth h}^{k}$.
\end{coro}

\begin{proof}
The corollary results from Lemma~\ref{lisc4} by the arguments of the proof of 
Corollary~\ref{cisc3}.
\end{proof}

\begin{prop}\label{pisc4}
There exists on $V^{(k)}$ a regular differential form of top degree without zero.
\end{prop}

\begin{proof}
As in the proof of Proposition~\ref{pisc3}, it suffices to prove that 
$m_{\alpha }=0$ for all $\alpha $ in ${\cal R}_{+}$ since 
$G.({\goth h}_{\r}^{2}\times {\goth h}^{k-2})$ is a dense open subset of 
$V^{(k)}$. As $V_{0,1,2}^{(k)}$ is contained in $V_{1,2}^{(k)}$, $m_{\alpha }=0$
by the proof of Proposition~\ref{pisc3}.  
\end{proof}

\subsection{} \label{isc5}
For $2\leq k'\leq k$, the projection
$$ \xymatrix{ {\cal X}^{k} \ar[rr] && {\cal X}^{k'}}, \qquad
(\poi x1{,\ldots,}{k}{}{}{},\poi y1{,\ldots,}{k}{}{}{}) \longmapsto 
(\poi x1{,\ldots,}{k'}{}{}{},\poi y1{,\ldots,}{k'}{}{}{})$$
induces the projections
$$ \xymatrix{ {\cal C}_{\x}^{(k)} \ar[rr] && {\cal C}_{\x}^{(k')}}, \qquad  
\xymatrix{ W_{j}^{(k)} \ar[rr] && W_{j}^{(k')}} .$$
Set:
$$ W_{1,2}^{(k)} := W_{1}^{(k)} \cap W_{2}^{(k)} .$$
According to Corollary~\ref{cisc2},(iii), $W_{1,2}^{(k)}$ is equal to 
$G.\iota _{\x,k}(V_{0,1,2}^{(k)})$.

\begin{lemma}\label{lisc5}
Let $\omega $ be a regular differential form of top degree on $W_{1}^{(k)}$, without 
zero. Denote by $\omega '$ its restriction to $W_{1,2}^{(k)}$.

{\rm (i)} For $\varphi $ in $\k[W_{1}^{(k)}]$, if $\varphi $ has no zero then $\varphi $
is in $\k^{*}$.

{\rm (ii)} For some invertible element $\psi $ of $\k[W_{1,2}^{(2)}]$, 
$\omega ' = \psi {\sigma _{2}}_{*}(\omega ')$.

{\rm (iii)} The function $\psi (\psi \rond \sigma _{2})$ on $W_{1,2}^{(k)}$ is equal 
to $1$. 
\end{lemma}

\begin{proof}
Following the arguments of the proof of Lemma~\ref{lisc3}, the lemma results from
Lemma~\ref{lisc2},(iii).
\end{proof}

\begin{coro}\label{cisc5}
The function $\psi $ is invariant under the action of $G$ in $W_{1,2}^{(k)}$ and for
some sequence $m_{\alpha },\alpha \in {\cal R}_{+}$ in ${\Bbb Z}$, 
$$ \psi \rond \iota _{\x,k}(\poi x1{,\ldots,}{k}{}{}{}) = 
\pm \prod_{\alpha \in {\cal R}_{+}}
(\alpha (x_{1})\alpha (x_{2})^{-1})^{m_{\alpha }},$$
for all $(\poi x1{,\ldots,}{k}{}{}{})$ in ${\goth h}_{\r}^{2}\times {\goth h}^{k}$.
\end{coro}

\begin{proof}
Since $W_{1,2}^{(k)}=G.\iota _{\x,k}(V_{0,1,2}^{(k)})$, the corollary results from 
Lemma~\ref{lisc5} by the arguments of the proof of Corollary~\ref{cisc3}. 
\end{proof}

\begin{prop}\label{pisc5}
There exists on $W^{(k)}$ a regular differential form of top degree without zero.
\end{prop}

\begin{proof}
As in the proof of Proposition~\ref{pisc3}, it suffices to prove that 
$m_{\alpha }=0$ for all $\alpha $ in ${\cal R}_{+}$ since 
$G.\iota _{\x,k}({\goth h}_{\r}^{2}\times {\goth h}^{k-2})$ is a dense open subset of 
$W^{(k)}$. As $W_{1,2}^{(k)} = G.\iota _{\x,k}(V_{0,1,2}^{(k)})$, $m_{\alpha }=0$
by the proof of Proposition~\ref{pisc3}.  
\end{proof}

\subsection{} \label{isc6}
Recall that $(G.X)_{\n}$ is the normalization of $G.X$. Denote by ${\cal E}_{\n}^{(k)}$ 
the following fiber products:
$$ \xymatrix{ {\cal E}_{\n}^{(k)} \ar[rr] ^{\nu _{\n,k}} \ar[d]_{\pi_{\n,k}}&& 
{\cal E}^{(k)} \ar[d]^{\pi } \\ (G.X)_{\n} \ar[rr]_{\nu } && G.X} $$
with $\nu $ the normalization morphism, $\pi _{\n,k}$, $\nu _{\n,k}$ the restriction 
maps. 

\begin{lemma}\label{lisc6}
The variety ${\cal E}_{\n}^{(k)}$ is the normalization of ${\cal E}^{(k)}$ and 
$\nu _{\n,k}$ is the normalization morphism. 
\end{lemma}

\begin{proof}
Since ${\cal E}^{(k)}$ is a vector bundle over $G.X$, ${\cal E}_{\n}^{(k)}$ is a vector 
bundle over $(G.X)_{\n}$. Then ${\cal E}_{\n}^{(k)}$ is normal since so is $(G.X)_{\n}$. 
Moreover, the fields of rational functions on ${\cal E}_{\n}^{(k)}$ and ${\cal E}^{(k)}$ 
are equal and the comorphism of $\nu _{\n,k}$ induces the morphism identity of this field
so that $\nu _{\n,k}$ is the normalization morphism. 
\end{proof}

Denote by $\widetilde{{\goth X}_{0,k}}$, $\widetilde{{\cal C}^{(k)}}$, 
$\widetilde{{\cal C}_{\x}^{(k)}}$ the normalizations of ${\goth X}_{0,k}$, 
${\cal C}^{(k)}$, ${\cal C}_{\x}^{(k)}$ respectively. Let $\lambda _{0,k}$, 
$\lambda _{k}$, $\lambda _{*,k}$ be the respective normalization morphisms.

\begin{lemma}\label{l2isc6}
{\rm (i)} There exists a projective birational morphism $\tau _{\n,0,k}$ from 
${\cal E}_{0}^{(k)}$ onto $\widetilde{{\goth X}_{0,k}}$ such that 
$\tau _{0,k}=\lambda _{0,k}\rond \tau _{\n,0,k}$. Moreover,
$\tau _{\n,0,k}^{-1}(\lambda _{0,k}^{-1}(V_{0}^{(k)}))$ is a smooth big open subset of 
${\cal E}_{0}^{(k)}$ and the restriction of $\tau _{\n,0,k}$ to this subset is an 
isomorphism onto $\lambda _{0,k}^{-1}(V_{0}^{(k)})$.

{\rm (ii)} There exists a projective birational morphism $\tau _{\n,k}$ from 
${\cal E}_{\n}^{(k)}$ onto $\widetilde{{\cal C}^{(k)}}$ such that 
$\tau _{k}\rond \nu _{\n,k}=\lambda _{k}\rond \tau _{\n,k}$. Moreover,
$\tau _{\n,k}^{-1}(\lambda _{k}^{-1}(V^{(k)}))$ is a smooth big open subset of 
${\cal E}_{\n}^{(k)}$ and the restriction of $\tau _{\n,k}$ to this subset is an 
isomorphism onto $\lambda _{k}^{-1}(V^{(k)})$.

{\rm (iii)} There exists a projective birational morphism $\tau _{\n,*,k}$ from 
$\sqx G{{\cal E}_{0}^{(k)}}$ onto $\widetilde{{\cal C}_{\x}^{(k)}}$ such that 
$\tau _{*,k}=\lambda _{*,k}\rond \tau _{\n,*,k}$. Moreover,
$\tau _{\n,*,k}^{-1}(\lambda _{*,k}^{-1}(W^{(k)}))$ is a smooth big open subset of 
$\sqx G{{\cal E}_{0}^{(k)}}$ and the restriction of $\tau _{\n,*,k}$ to this subset is an 
isomorphism onto $\lambda _{*,k}^{-1}(W^{(k)})$.
\end{lemma}

\begin{proof}
(i) According to Corollary~\ref{cisc2},(i), $\tau _{0,k}$ is a
birational morphism from ${\cal E}_{0}^{(k)}$ onto ${\goth X}_{0,k}$ and 
${\cal E}_{0}^{(k)}$ is a normal variety since so is $X$ by~\cite[Theorem 1.1]{C1}. Hence
it factorizes through $\lambda _{0,k}$ so that for some birational morphism 
$\tau _{\n,0,k}$ from ${\cal E}_{0}^{(k)}$ to $\widetilde{{\goth X}_{0,k}}$, 
$\tau _{0,k}=\lambda _{0,k}\rond \tau _{\n,0,k}$, whence the 
commutative diagram:
$$ \xymatrix{ & \ar[ld]_{\tau _{\n,0,k}} 
{\cal E}_{0}^{(k)} \ar[rd]^{\tau _{0,k}} &  \\ 
\widetilde{{\goth X}_{0,k}} \ar[rr]_{\lambda _{0,k}} && {\goth X}_{0,k}} .$$
According to Lemma~\ref{lisc1},(i), $\tau _{0,k}$ is a projective morphism. Hence 
so is $\tau _{\n,0,k}$ since it deduces from $\tau _{0,k}$ by base extension
\cite[Ch. II, Exercise 4.9]{Ha}. 

According to Lemma~\ref{l2isc2},(i), $\tau _{0,k}^{-1}(V_{0}^{(k)})$ is a big open subset
of ${\cal E}_{0}^{(k)}$. Moreover, we have the commutative diagram
$$ \xymatrix{ & \tau _{0,k}^{-1}(V_{0}^{(k)}) \ar[ld]_{\tau _{\n,0,k}}  
\ar[rd]^{\tau _{0,k}} \\ \lambda _{0,k}^{-1}(V_{0}^{(k)}) \ar[rr]_{\lambda _{0,k}} &&
V_{0}^{(k)}} .$$ 
By Lemma~\ref{lisc2},(i), $V_{0}^{(k)}$ is a smooth open subset of ${\goth X}_{0,k}$
so that $\lambda _{0,k}$ is an isomorphism from $\lambda _{0,k}^{-1}(V_{0}^{(k)})$ 
onto $V_{0}^{(k)}$. By Corollary~\ref{cisc2},(i), $\tau _{0,k}$ is an isomorphism 
from $\tau _{0,k}^{-1}(V_{0}^{(k)})$ onto $V_{0}^{(k)}$ so that 
$\tau _{0,k}^{-1}(V_{0}^{(k)})$ is a smooth open subset of ${\cal E}_{0}^{(k)}$. As a 
result, $\tau _{\n,0,k}$ is an isomorphism from $\tau _{0,k}^{-1}(V_{0}^{(k)})$ onto 
$\lambda _{0,k}^{-1}(V_{0}^{(k)})$.

(ii) According to Corollary~\ref{cisc2},(ii), $\tau _{k}\rond \nu _{\n,k}$ is a
birational morphism from ${\cal E}_{\n}^{(k)}$ onto ${\cal C}^{(k)}$ and ${\cal E}_{\n}^{(k)}$ is
a normal variety by Lemma~\ref{lisc6},(i). Hence it factorizes
through $\lambda _{k}$ so that for some birational morphism $\tau _{\n,k}$ from 
${\cal E}_{\n}^{(k)}$ to $\widetilde{{\cal C}^{(k)}}$, 
$\tau _{k}\rond \nu _{\n,k}=\lambda _{k}\rond \tau _{\n,k}$, whence the 
commutative diagram:
$$ \xymatrix{ {\cal E}_{\n}^{(k)} \ar[rr]^{\nu _{\n,k}} \ar[d]_{\tau _{\n,k}} &&
{\cal E}^{(k)} \ar[d]^{\tau _{k}} \\ \widetilde{{\cal C}^{(k)}} \ar[rr]_{\lambda _{k}}
&& {\cal C}^{(k)}} .$$
According to Lemma~\ref{lisc1},(i), $\tau _{k}$ is a projective morphism. Hence 
so is $\tau _{\n,k}$ since it deduces from $\tau _{k}$ by base extension
\cite[Ch. II, Exercise 4.9]{Ha}. 

According to Lemma~\ref{l2isc2},(ii), 
$\tau _{\n,k}^{-1}(\lambda _{k}^{-1}(V^{(k)}))$ is a big open subset of 
${\cal E}_{\n}^{(k)}$ since $\nu _{\n,k}$ is a finite morphism. Moreover, we have the 
commutative diagram
$$ \xymatrix{ \tau _{\n,k}^{-1}(\lambda _{k}^{-1}(V^{(k)})) 
\ar[rr]^{\nu _{\n,k}} \ar[d]_{\tau _{\n,k}} && \tau _{k}^{-1}(V^{(k)}) 
\ar[d]^{\tau _{k}} \\ \lambda _{k}^{-1}(V^{(k)}) \ar[rr]_{\lambda _{k}} &&
V^{(k)}} .$$ 
By Lemma~\ref{lisc2},(ii), $V^{(k)}$ is a smooth open subset of ${\cal C}^{(k)}$
so that $\lambda _{k}$ is an isomorphism from $\lambda _{k}^{-1}(V^{(k)})$ 
onto $V^{(k)}$. By Corollary~\ref{cisc2},(ii), $\tau _{k}$ is an isomorphism 
from $\tau _{k}^{-1}(V^{(k)})$ onto $V^{(k)}$ so that $\tau _{k}^{-1}(V^{(k)})$ is
a smooth open subset of ${\cal E}^{(k)}$ and $\nu _{\n,k}$ is an isomorphism from
$\tau _{\n,k}^{-1}(\lambda _{k}^{-1}(V^{(k)}))$ onto 
$\tau _{k}^{-1}(V^{(k)})$. As a result, $\tau _{\n,k}$ is an isomorphism 
from $\tau _{\n,k}^{-1}(\lambda _{k}^{-1}(V^{(k)}))$ onto 
$\lambda _{k}^{-1}(V^{(k)})$ and 
$\tau _{\n,k}^{-1}(\lambda _{k}^{-1}(V^{(k)}))$ is a smooth open subset of 
${\cal E}_{\n}^{(k)}$.
 
(iii) According to Corollary~\ref{cisc2},(iii), $\tau _{*,k}$ is a
birational morphism from $\sqx G{{\cal E}_{0}^{(k)}}$ onto ${\cal C}_{\x}^{(k)}$ and 
$\sqx G{{\cal E}_{0}^{(k)}}$ is a normal variety as a vector bundle over 
$\sqx GX$ which is normal by Proposition~\ref{pmv}. Hence it factorizes through 
$\lambda _{*,k}$ so that for some birational morphism $\tau _{\n,*,k}$ from 
$\sqx G{{\cal E}_{0}^{(k)}}$ to $\widetilde{{\cal C}_{\x}^{(k)}}$, 
$\tau _{*,k}=\lambda _{*,k}\rond \tau _{\n,*,k}$, whence the commutative diagram:
$$ \xymatrix{ & \ar[ld]_{\tau _{\n,*,k}} 
\sqx G{{\cal E}_{0}^{(k)}} \ar[rd]^{\tau _{*,k}} & \\ \widetilde{{\cal C}_{\x}^{(k)}} 
\ar[rr]_{\lambda _{*,k}} && {\cal C}_{\x}^{(k)}} .$$
According to Lemma~\ref{lisc1},(i), $\tau _{*,k}$ is a projective morphism. Hence 
so is $\tau _{\n,*,k}$ since it deduces from $\tau _{*,k}$ by base extension
\cite[Ch. II, Exercise 4.9]{Ha}. 

According to Lemma~\ref{l2isc2},(iii), $\tau _{*,k}^{-1}(W^{(k)})$ is a big open subset 
of $\sqx G{{\cal E}_{0}^{(k)}}$. Moreover, we have the commutative diagram
$$ \xymatrix{ & \tau _{*,k}^{-1}(W^{(k)}) \ar[ld]_{\tau _{\n,*,k}} 
\ar[rd]^{\tau _{*,k}} &  \\ \lambda _{*,k}^{-1}(W^{(k)}) \ar[rr]_{\lambda _{*,k}} &&
W^{(k)}} .$$ 
By Lemma~\ref{lisc2},(iii), $W^{(k)}$ is a smooth open subset of ${\cal C}_{\x}^{(k)}$
so that $\lambda _{*,k}$ is an isomorphism from $\lambda _{*,k}^{-1}(W^{(k)})$ 
onto $W^{(k)}$. By Corollary~\ref{cisc2},(i), $\tau _{*,k}$ is an isomorphism 
from $\tau _{*,k}^{-1}(W^{(k)})$ onto $W^{(k)}$ so that $\tau _{*,k}^{-1}(W^{(k)})$ is
a smooth open subset of $\sqx G{{\cal E}_{0}^{(k)}}$. As a result, $\tau _{\n,*,k}$ is an
isomorphism from $\tau _{*,k}^{-1}(W^{(k)})$ onto $\lambda _{*,k}^{-1}(W^{(k)})$.
\end{proof}

Let ${\goth Y}$ be one of the three varieties $\widetilde{{\goth X}}_{0,k}$, 
$\widetilde{{\cal C}^{(k)}}$, $\widetilde{{\cal C}_{\x}^{(k)}}$ and set:
$$ {\goth Z} := \left \{ \begin{array}{ccc} {\cal E}_{0}^{(k)} & \mbox{ if } & 
{\goth Y} = \widetilde{{\goth X}_{0,k}} \\
{\cal E}_{\n}^{(k)} & \mbox{ if } & {\goth Y} = \widetilde{{\cal C}^{(k)}} \\
\sqx G{{\cal E}_{0}^{(k)}} & \mbox{ if } & {\goth Y} = \widetilde{{\cal C}_{\x}^{(k)}} 
\end{array} \right. , \qquad \tauup := \left \{ \begin{array}{ccc}
\tau _{\n,0,k} & \mbox{ if } & {\goth Y} = \widetilde{{\goth X}_{0,k}} \\
\tau _{\n,k} & \mbox{ if } & {\goth Y} = \widetilde{{\cal C}^{(k})} \\
\tau _{\n,*,k} & \mbox{ if } & {\goth Y} = \widetilde{{\cal C}_{\x}^{(k})} \end{array}
\right. ,$$
$$ {\goth T} := \left \{ \begin{array}{ccc} 
X & \mbox{ if } & {\goth Y} = \widetilde{{\goth X}_{0,k}} \\ 
G.X_{\n} & \mbox{ if } & {\goth Y} = \widetilde{{\cal C}^{(k})} \\ 
\sqx G{X} & \mbox{ if } & {\goth Y} = \widetilde{{\cal C}_{\x}^{(k})} \end{array}
\right. , \qquad \piup := \left \{ \begin{array}{ccc}
\xymatrix{ {\cal E}_{0}^{(k)} \ar[r] & X} & \mbox{ if } &
{\goth Y} = \widetilde{{\goth X}_{0,k}} \\
\xymatrix{ {\cal E}_{\n}^{(k)} \ar[r] & (G.X)_{\n}} & \mbox{ if } &
{\goth Y} = \widetilde{{\cal C}^{(k)}} \\
\xymatrix{ \sqx G{{\cal E}_{\n}^{(k)}} \ar[r] & \sqx G{X}} & \mbox{ if } &
{\goth Y} = \widetilde{{\cal C}_{\x}^{(k)}} \end{array}  \right. ,$$
where the arrow is the bundle projection in the last three equalities.

\begin{prop}\label{pisc6}
{\rm (i)} The morphism $\tauup $ is a projective birational morphism.

{\rm (ii)} The set ${\goth Z}_{\loc}$ is the inverse image of ${\goth T}_{\loc}$ by 
$\piup $.

{\rm (iii)} For some smooth big open subset ${\goth O}$ of ${\goth Z}_{\loc}$, the 
restriction of $\tauup $ to ${\goth O}$ is an isomorphism onto a smooth big open subset
of ${\goth Y}$.

{\rm (iv)} The sheaves $\Omega _{{\goth Y}_{\loc}}$ and $\Omega _{{\goth Z}_{\loc}}$ have 
a global section without $0$.
\end{prop}

\begin{proof}
(i) The assertion results from Lemma~\ref{l2isc6}.

(ii) As a polynomial algebra over an algebra $A$ is regular if and only if so is $A$, 
${\goth Z}_{\loc} = \piup ^{-1}({\goth T}_{\loc})$ since ${\goth Z}$ is a vector bundle 
over ${\goth T}$. 

(iii) The assertion results from Lemma~\ref{l2isc6} and Lemma~\ref{l2isc2}.

(iv) For ${\goth Y}=\widetilde{{\goth X}_{0,k}}$, the assertion results from 
Lemma~\ref{lars}, Proposition~\ref{pisc3}, Lemma~\ref{lisc2},(i) and 
Lemma~\ref{l2isc2},(i). For ${\goth Y}=\widetilde{{\cal C}^{(k)}}$, the assertion results
from Lemma~\ref{lars}, Proposition~\ref{pisc4}, Lemma~\ref{lisc2},(ii) and 
Lemma~\ref{l2isc2},(ii). For ${\goth Y}=\widetilde{{\cal C}_{\x}^{(k)}}$, the assertion 
results from Lemma~\ref{lars}, Proposition~\ref{pisc5}, Lemma~\ref{lisc2},(iii) and 
Lemma~\ref{l2isc2},(iii).
\end{proof}

\section{Rational singularities} \label{rs}
Let $k\geq 2$ be an integer and let ${\goth Y}$, ${\goth Z}$, ${\goth T}$, 
$\tauup $, $\piup $ be as in Proposition~\ref{pisc6}. Denote by $\iota $ the canonical
embeddings $\xymatrix{ {\goth Y}_{\loc} \ar[r] & {\goth Y}}$ and 
$\xymatrix{ {\goth Z}_{\loc} \ar[r] & {\goth Z}}$. According to~\cite{Hi}, there exists a
desingularization $\Gamma $ of ${\goth T}$ with morphism $\thetaup $ such that the 
restriction of $\thetaup $ to $\thetaup ^{-1}({\goth T}_{\loc})$ is an isomorphism onto 
${\goth T}_{\loc}$. Let $\widetilde{{\goth Z}}$ be the following fiber product
$$ \xymatrix{ \widetilde{{\goth Z}} \ar[rr]^{\overline{\thetaup }} 
\ar[d]_{\overline{\piup }} && {\goth Z} \ar[d]^{\piup }\\ 
\Gamma \ar[rr]_{\thetaup } && {\goth T}}$$
with $\overline{\thetaup }$ and $\overline{\pi }$ the restriction maps so that 
$\widetilde{{\goth Z}}$ is a vector bundle of rank $k\rg$ over $\Gamma $ and 
$\overline{\piup }$ is the bundle projection. Moreover, $\overline{\thetaup }$ is 
projective and birational so that $\widetilde{{\goth Z}}$ is a desingularization of 
${\goth Z}$ and ${\goth Y}$ by Proposition~\ref{pisc6},(i). 

\begin{prop}\label{prs}
Suppose ${\goth Y}=\widetilde{{\goth X}_{0,k}}$ or $\widetilde{{\cal C}_{\x}^{(k)}}$.

{\rm (i)} The variety ${\goth Z}$ is Gorenstein with rational singularities. Moreover,
its canonical bundle is free of rank one.

{\rm (ii)} The variety ${\goth Y}$ is Gorenstein with rational singularities. Moreover,
its canonical bundle is free of rank one.
\end{prop}

\begin{proof}
(i) According to~\cite[Theorem 1.1]{C1}, $X$ has rational singularities. Then, by 
Lemma~\ref{lsi},(iii), so has $\sqx GX$ as a fiber bundle over a smooth variety with 
fibers having rational singularities. As a result, by Lemma~\ref{lsi},(iv), ${\goth Z}$ 
has rational singularities as a vector bundle over a variety having rational 
singularities. Moreover, ${\goth T}$ is Gorenstein by Proposition~\ref{pmv},(i) and 
(iii). Then so is ${\goth Z}$ as a vector bundle over ${\goth T}$ by 
Lemma~\ref{lsi},(i). By Proposition~\ref{pisc6}, $\Omega _{{\goth Z}_{\loc}}$ has a 
global section without zero. Then, by Lemma~\ref{l2ars}, 
$\iota _{*}(\Omega _{{\goth Z}_{\loc}})$ is a free module of rank one. Since ${\goth Z}$ 
has rational singularities, the canonical module of ${\goth Z}$ is equal to  
$\iota _{*}(\Omega _{{\goth Z}_{\loc}})$ by~\cite[p.50]{KK}, whence the assertion.

(ii) By Proposition~\ref{pisc6}, $\Omega _{{\goth Y}_{\loc}}$ has a global 
section without zero. Denote it by $\omega $. By Proposition~\ref{pisc6},(iii), 
$\tauup ^{*}(\omega )$ is a local section of $\Omega _{{\goth Z}_{\loc}}$ above a big 
open subset of ${\goth Z}$. So by (i) and ~\cite[p.50]{KK}, $\tauup ^{*}(\omega )$ has a 
regular extension to $\widetilde{{\goth Z}}$. Denote it by $\widetilde{\omega }$ and 
by $\mu $ the morphism 
$$\xymatrix{ \an {\widetilde{{\goth Z}}}{} \ar[rr]^{\mu } && 
\Omega _{\widetilde{{\goth Z}}}}, 
\qquad \varphi \longmapsto \varphi \widetilde{\omega } .$$
Since $\omega $ has no zero, by Lemma~\ref{l2ars}, 
$\tauup \rond \overline{\thetaup }_{*}\mu $ is an isomorphism from $\an {{\goth Y}}{}$ 
onto $\Omega _{{\goth Y}_{\loc}}$ and $\iota _{*}(\Omega _{{\goth Y}_{\loc}})$ is a free 
module of rank one. As a result, by~\cite[Lemma 2.3]{Hin}, ${\goth Y}$ is Gorenstein with
rational singularities. Then, by~\cite[p.50]{KK}, the canonical module of ${\goth Y}$ is 
equal to $\iota _{*}(\Omega _{{\goth Y}_{\loc}})$, whence the assertion. 
\end{proof}

\begin{coro}\label{crs}
{\rm (i)} The variety $\widetilde{{\cal C}^{(k)}}$ is Gorenstein with rational 
singularities. Moreover its canonical module is free of rank one.

{\rm (ii)} The variety ${\cal E}_{\n}^{(k)}$ is Gorenstein with rational 
singularities. Moreover its canonical module is free of rank one.

{\rm (iii)} The varieties ${\cal E}_{\n}$ and $(G.X)_{\n}$ are Gorenstein with rational 
singularities.
\end{coro}

\begin{proof}
In the proof, we suppose ${\goth Y}=\widetilde{{\cal C}^{(k)}}$.

(i) According to~\cite[Proposition 5.8,(ii)]{CZ}, $\widetilde{{\cal C}^{(k)}}$ is the 
categorical quotient of $\widetilde{{\cal C}_{\x}^{(k)}}$ by the action of $W({\cal R})$.
Hence, by \cite[Th\'eor\`eme]{Bout} and Proposition~\ref{prs},(ii), 
${\goth Y}=\widetilde{{\cal C}^{(k)}}$ has rational singularities. By Proposition
\ref{pisc6},(iv), $\Omega _{{\goth Y}_{\loc}}$ has a global section without zero. 
Then, by Lemma~\ref{l2ars}, $\iota _{*}(\Omega _{{\goth Y}_{\loc}})$ is a free module of 
rank one. Since ${\goth Y}$ has rational singularities, the canonical module of 
${\goth Y}$ is equal to $\iota _{*}(\Omega _{{\goth Z}_{\loc}})$ by~\cite[p.50]{KK}. 
Moreover, ${\goth Y}$ is Cohen-Macaulay. So, by Lemma~\ref{l3ars},
$\iota _{*}(\Omega _{{\goth Z}_{\loc}})$ has finite injective dimension, whence
${\goth Y}$ is Gorenstein.

(ii) Denote by $\omega $ a global section of $\Omega _{{\goth Y}_{\loc}}$ without zero. 
By Proposition~\ref{pisc6},(iii) and Lemma~\ref{l2ars}, $\Omega _{{\goth Z}_{\loc}}$ has
a global section without zero whose restriction to a big open subset of 
${\goth Z}_{\loc}$ is equal to the restriction of $\tauup ^{*}(\omega )$. Denote it by 
$\omega '$. Since ${\goth Y}$ has rational singularities, 
$\tauup \rond \overline{\thetaup }^{*}(\omega )$ has a regular extension 
to $\widetilde{{\goth Z}}$ by~\cite[p.50]{KK}. Denote it by $\widetilde{\omega }$.
Then the restriction of $\widetilde{\omega }$ to 
$\overline{\thetaup }^{-1}({\goth Z}_{\loc})$ is equal to 
$\overline{\thetaup }^{*}(\omega ')$. Let $\mu $ be the morphism 
$$\xymatrix{ \an {\widetilde{{\goth Z}}}{} \ar[rr]^{\mu } && 
\Omega _{\widetilde{{\goth Z}}}}, 
\qquad \varphi \longmapsto \varphi \widetilde{\omega } .$$
Since $\omega '$ has no zero, by Lemma~\ref{l2ars}, 
$\overline{\thetaup }_{*}\mu $ is an isomorphism from $\an {{\goth Z}}{}$ 
onto $\Omega _{{\goth Z}_{\loc}}$ and $\iota _{*}(\Omega _{{\goth Z}_{\loc}})$ is free of
rank one. As a result, by~\cite[Lemma 2.3]{Hin}, ${\goth Z}$ is Gorenstein with rational 
singularities. Then, by~\cite[p.50]{KK}, the canonical module of ${\goth Z}$ is equal to 
$\iota _{*}(\Omega _{{\goth Z}_{\loc}})$, whence the assertion.

(iii) Since ${\goth Z}$ is a vector bundle over ${\goth T}=(G.X)_{\n}$, $(G.X)_{\n}$
is Gorenstein with rational singularities by (ii) and Lemma~\ref{lsi},(ii) and (iv). Then
so is ${\cal E}_{\n}$ as a vector bundle over $(G.X)_{\n}$ by 
Lemma~\ref{lsi},(i) and (iv).
\end{proof}

Summarizing the results, Theorem~\ref{tint} results from Proposition~\ref{prs},(ii),
and Corollary~\ref{crs},(i) and (iii). According to~\cite{Ric}, 
${\cal C}^{(2)}$ is the commuting variety of ${\goth g}$ and according 
to~\cite[Theorem 1.1]{C}, ${\cal C}^{(2)}$ is normal, whence:

\begin{coro}\label{c2rs}
The commuting variety of ${\goth g}$ is Gorenstein with rational singularities. 
Moreover, its canonical module is free of rank $1$.
\end{coro}

\section{Normality} \label{nor}
Let $k$ be a positive integer. The goal of this section is to prove that ${\goth X}_{0,k}$
is a normal variety. Consider the desingularization $(\Gamma ,\thetaup )$ of $X$ as in 
Section~\ref{rs}. For simplicity of the notations, for $k$ positive integer, we denote
by $\piup _{k}$ the bundle projection $\xymatrix{ {\cal E}_{0}^{(k)}\ar[r] & X}$ and  by 
$F^{(k)}$ the fiber product
$$ \xymatrix{ F^{(k)} \ar[d]_{\overline{\piup _{k}}} \ar[rr]^{\thetaup _{k}} && 
{\cal E}_{0}^{(k)} \ar[d]^{\piup _{k}} \\ \Gamma \ar[rr]_{\thetaup } && X}$$ 
with $\thetaup _{k}$ and $\overline{\piup _{k}}$ the restriction morphisms.

\subsection{} \label{nor1}
Let $F^{*}$ be the dual of the vector bundle $F^{(1)}$ over $\Gamma $. 

\begin{lemma}\label{lnor1}
Let ${\cal F}^{*}$ be the sheaf of local sections of $F^{*}$. For $i>0$ and 
for $j\geq 0$, ${\mathrm {H}}^{i}(\Gamma ,\sy j{{\cal F}^{*}})=0$.
\end{lemma}

\begin{proof}
Since $\overline{\piup _{1}}$ is the bundle projection of the vector bundle $F^{(1)}$ 
over $\Gamma $, $\an {F^{(1)}}{}$ is equal to 
$\overline{\piup _{1}}^{*}(\es S{{\cal F}^{*}})$ so that
$$ \overline{\piup _{1}}_{*}(\an {F}{}) = \es S{{\cal F}^{*}}$$
As a result, for $i\geq 0$,
$$ {\mathrm {H}}^{i}(F^{(1)},\an {F^{(1)}}{}) = 
{\mathrm {H}}^{i}(\Gamma ,\es S{{\cal F}^{*}}) = 
\bigoplus _{j\in {\Bbb N}} {\mathrm {H}}^{i}(\Gamma ,\sy j{{\cal F}^{*}}) $$
According to Lemma~\ref{lmv},(i), $F^{(1)}$ is a desingularization of the 
smooth variety ${\goth b}$. Hence by~\cite{El}, 
$${\mathrm {H}}^{i}(F^{(1)},\an {F^{(1)}}{})=0$$ for $i>0$, whence 
$${\mathrm {H}}^{i}(\Gamma ,\sy j{{\cal F}^{*}})=0$$ 
for $i>0$ and $j\geq 0$. 
\end{proof}

According to the identification of ${\goth g}$ and ${\goth g}^{*}$ by the bilinear form
$\dv ..$, ${\goth b}_{-}$ identifies with ${\goth b}^{*}$. Denote by $F_{-}$ the 
orthogonal complement to $F^{(1)}$ in $\Gamma \times {\goth b}_{-}$ so that $F_{-}$ is
a vector bundle of rank $n$ over $\Gamma $. Let ${\cal F}_{-}$ be the sheaf of 
local sections of $F_{-}$.

\begin{coro}\label{cnor1}
Let ${\cal J}_{0}$ be the ideal of $\tk {\k}{\an {\Gamma }{}}{\ec Sb{}-{}}$ generated 
by ${\cal F}_{-}$. Then, for $i\geq 0$, ${\mathrm {H}}^{i}(\Gamma ,{\cal J}_{0})=0$
and ${\mathrm {H}}^{i}(\Gamma ,{\cal F}_{-})=0$.
\end{coro}

\begin{proof}
Since $F_{-}$ is the orthogonal complement to $F^{(1)}$ in 
$\Gamma \times {\goth b}_{-}$, ${\cal J}_{0}$ is the ideal of definition of
$F^{(1)}$ in $\tk {\k}{\an {\Gamma }{}}{\es S{{\goth b}_{-}}}$ whence a short 
exact sequence
$$ 0 \longrightarrow {\cal J}_{0} \longrightarrow 
\tk {\k}{\an {\Gamma }{}}{\ec Sb{}-{}} \longrightarrow \es S{{\cal F}^{*}} 
\longrightarrow 0$$
and whence a cohomology long exact sequence
$$ \cdots \longrightarrow {\mathrm {H}}^{i}(\Gamma ,\es S{{\cal F}^{*}})
\longrightarrow {\mathrm {H}}^{i+1}(\Gamma ,{\cal J}_{0}) 
\longrightarrow 
{\mathrm {H}}^{i+1}(\Gamma ,\tk {\k}{\an {\Gamma }{}}{\ec Sb{}-{}}) \longrightarrow 
\cdots .$$
Then, by Lemma~\ref{lnor1}, from the equality
$$ {\mathrm {H}}^{i}(\Gamma ,\tk {\k}{\an {\Gamma }{}}{\ec Sb{}-{}}) = 
\tk {\k}{\es S{{\goth b}_{-}}}{\mathrm {H}}^{i}(\Gamma ,\an {\Gamma }{})$$
for all $i$, we deduce ${\mathrm {H}}^{i}(\Gamma ,{\cal J}_{0})=0$ for 
$i\geq 2$. Moreover, since $\Gamma $ is an irreducible projective variety, 
${\mathrm {H}}^{0}(\Gamma ,\an {\Gamma }{})=\k$ and since $F^{(1)}$ is a desingularization of 
${\goth b}$, 
${\mathrm {H}}^{0}(\Gamma ,\es S{{\cal F}^{*}})=\es S{{\goth b}_{-}}$ so that the
map
$${\mathrm {H}}^{0}(\Gamma ,\tk {\k}{\an {\Gamma }{}}\es S{{\goth b}_{-}})
\longrightarrow {\mathrm {H}}^{0}(\Gamma , \es S{{\cal F}^{*}})$$
is an isomorphism. Hence ${\mathrm {H}}^{i}(\Gamma ,{\cal J}_{0})=0$ for $i=0,1$.
The gradation on $\ec Sb{}-{}$ induces a gradation on 
$\tk {\k}{\an {\Gamma }{}}{\ec Sb{}-{}}$ so that ${\cal J}_{0}$ is a graded ideal. Since 
${\cal F}_{-}$ is the subsheaf of local sections of degree $1$ of ${\cal J}_{0}$, 
it is a direct factor of ${\cal J}_{0}$, whence the corollary.
\end{proof}

\begin{prop}\label{pnor1}
Let $k,l$ be nonnegative integers. 

{\rm (i)} For all positive integer $i$,
${\mathrm {H}}^{i}(\Gamma ,({\cal F}^{*})^{\tens k})=0$.

{\rm (ii)} For all positive integer $i$,
$$ {\mathrm {H}}^{i+l}(\Gamma ,\tk {\an {\Gamma }{}}{{\cal F}_{-}^{\tens l}}
({\cal F}^{*})^{\tens k}) = 0 .$$
\end{prop}

\begin{proof}
(i) According to Lemma~\ref{lnor1}, we can suppose $k>1$. Denote by $F^{*}_{k}$ the 
restriction to the diagonal of $\Gamma ^{k}$ of the vector bundle ${F^{*}}^{k}$ over
$\Gamma ^{k}$. Identifying $\Gamma $ with the diagonal of $\Gamma ^{k}$, 
$F^{*}_{k}$ is a vector bundle over $\Gamma $. Since $F^{*}$ is the dual of 
the vector bundle $F^{(1)}$ over $\Gamma $, $F_{k}^{*}$ is the dual of the vector
bundle $F^{(k)} $ over $\Gamma $. Let $\psi _{k}$ be the bundle projection of
$F_{k}^{*}$ and let ${\cal F}_{k}^{*}$ be the sheaf of local sections of 
$F_{k}^{*}$. Then $\an {F^{(k)}}{}$ is equal to 
$\psi _{k}^{*}(\es S{{\cal F}_{k}^{*}})$ and since $F^{(k)}$ is a vector bundle over 
$\Gamma $, for all nonnegative integer $i$,
$${\mathrm {H}}^{i}(F^{(k)},\an {F^{(k)}}{}) = 
{\mathrm {H}}^{i}(\Gamma ,\es S{{\cal F}^{*}_{k}}) =
\bigoplus _{q\in {\Bbb N}} {\mathrm {H}}^{i}(\Gamma ,\sy q{{\cal F}^{*}_{k}}) .$$
According to Proposition~\ref{prs},(ii), for $i>0$, the left hand side is equal to $0$ 
since $F^{(k)}$ is a desingularization of $\widetilde{{\goth X}_{0,k}}$ by 
Proposition~\ref{pisc6},(i). As a result, for $i>0$,
$$ {\mathrm {H}}^{i}(\Gamma,\sy k{{\cal F}_{k}^{*}})) = 0 .$$
The decomposition of ${\cal F}_{k}^{*}$ as a direct sum of $k$ copies isomorphic to
${\cal F}^{*}$ induces a multigradation of $\es S{{\cal F}_{k}^{*}}$. Denoting  
by ${\cal S}_{\poi j1{,\ldots,}{k}{}{}{}}$ the subsheaf of multidegree 
$(\poi j1{,\ldots,}{k}{}{}{})$, we have
$$ \sy k{{\cal F}_{k}^{*}} = \bigoplus _{(\poi j1{,\ldots,}{k}{}{}{})\in {\Bbb N}^{k}
\atop {\poi j1{+\cdots +}{k}{}{}{}}=k} {\cal S}_{\poi j1{,\ldots,}{k}{}{}{}} 
\quad  \text{and} \quad {\cal S}_{1,\ldots,1} = ({\cal F}^{*})^{\tens k} .$$
Hence for $i>0$,
$$ 0 = {\mathrm {H}}^{i}(\Gamma ,\sy k{{\cal S}_{k}^{*}}) = 
\bigoplus _{(\poi j1{,\ldots,}{k}{}{}{})\in {\Bbb N}^{k}
\atop {\poi j1{+\cdots +}{k}{}{}{}}=k} 
{\mathrm {H}}^{i}(\Gamma ,{\cal S}_{\poi j1{,\ldots,}{k}{}{}{}})$$
whence the assertion.

(ii) Let $k$ be a nonnegative integer. Prove by induction on $j$ that for $i>0$ and 
for $l\geq j$,
\begin{eqnarray}\label{eqnor1}
{\mathrm {H}}^{i+j}(\Gamma ,
\tk {\an {\Gamma }{}}{{\cal F}_{-}^{\tens j}}({\cal F}^{*})^{\tens (k+l-j)}) =0 .
\end{eqnarray}
By (i) it is true for $j=0$. Suppose $j>0$ and (\ref{eqnor1}) true for $j-1$ and for
all $l\geq j-1$. From the short exact sequence of $\an {\Gamma }{}$-modules
$$ 0 \longrightarrow {\cal F}_{-} \longrightarrow \tk {\k}{\an {\Gamma }{}}{\goth b}_{-} 
\longrightarrow {\cal F}^{*} \longrightarrow 0$$
we deduce the short exact sequence of $\an {\Gamma }{}$-modules 
$$ 0 \longrightarrow 
\tk {\an {\Gamma }{}}{{\cal F}_{-}^{\tens j}}({\cal F}^{*})^{\tens (k+l-j)} 
\longrightarrow \tk {\k}{{\goth b}_{-}}\tk {\an {\Gamma }{}}{{\cal F}_{-}^{\tens (j-1)}}
({\cal F}^{*})^{\tens (k+l-j)} \longrightarrow 
\tk {\an {\Gamma }{}}{{\cal F}_{-}^{\tens (j-1)}}({\cal F}^{*})^{\tens (k+l-j+1)} 
\longrightarrow 0 .$$
From the cohomology long exact sequence deduced from this short exact sequence, we have 
the exact sequence
\begin{eqnarray*}
{\mathrm {H}}^{i+j-1}(\Gamma ,
\tk {\an {\Gamma }{}}{{\cal F}_{-}^{\tens (j-1)}}({\cal F}^{*})^{\tens (k+l-j+1)})
\longrightarrow {\mathrm {H}}^{i+j}(\Gamma ,
\tk {\an {\Gamma }{}}{{\cal F}_{-}^{\tens j}}({\cal F}^{*})^{\tens (k+l-j)}) \\ 
\longrightarrow 
{\mathrm {H}}^{i+j}(\Gamma ,\tk {\k}{{\goth b}_{-}}
\tk {\an {\Gamma }{}}{{\cal F}_{-}^{\tens (j-1)}}({\cal F}^{*})^{\tens (k+l-j)})
\end{eqnarray*}
for all positive integer $i$. By induction hypothesis, the first term equals $0$ for 
all $i>0$. Since 
$${\mathrm {H}}^{i+j}(\Gamma ,\tk {\k}{{\goth b}_{-}}
\tk {\an {\Gamma }{}}{{\cal F}_{-}^{\tens (j-1)}}({\cal F}^{*})^{\tens (k+l-j)}) =
\tk {\k}{{\goth b}_{-}}{\mathrm {H}}^{i+j}(\Gamma ,
\tk {\an {\Gamma }{}}{{\cal F}_{-}^{\tens (j-1)}}({\cal F}^{*})^{\tens (k+l-j)}),$$
the last term of the last exact sequence equals $0$ by induction hypothesis again, 
whence Equality (\ref{eqnor1}) and whence the assertion for $j=l$.
\end{proof}

The following corollary results from Proposition~\ref{pnor1},(ii) and 
Proposition~\ref{pco1}.

\begin{coro}\label{c2nor1}
For $k$ positive integer and for $l=(\poi l1{,\ldots,}{k}{}{}{})$ in ${\Bbb N}^{k}$,
$${\mathrm {H}}^{i+\vert l \vert}(\Gamma ,\tk {\an {\Gamma }{}}
{\ex {l_{1}}{{\cal F}_{-}}}\tk {\an {\Gamma }{}}{\cdots }\ex {l_{k}}{{\cal F}_{-}}) = 0$$
for all positive integer $i$.
\end{coro}

\subsection{} \label{nor2}
By definition, $F^{(k)}$ is a closed subvariety of $\Gamma \times {\goth b}^{k}$. Denote 
by $\varrho $ the canonical projection from $\Gamma \times {\goth b}^{k}$ to 
$\Gamma $, whence the diagram
$$\xymatrix{ F^{(k)} \ar@{^{(}->}[r] \ar[rd]_{\overline{\piup _{k}}}
&  \Gamma \times {\goth b}^{k} \ar[d]^{\varrho } \\
&  \Gamma }$$
For $j=1,\ldots,k$, denote by ${\goth S}_{j,k}$ the set of injections from 
$\{1,\ldots,j\}$ to $\{1,\ldots,k\}$ and for $\sigma $ in ${\goth S}_{j,k}$, set:
$$ {\cal K}_{\sigma } := \tk {\an {\Gamma }{}}{{\cal M}_{1}}
\tk {\an {\Gamma }{}}{\cdots }{\cal M}_{k} \mbox{ with }
{\cal M}_{i} := \left \{ \begin{array}{ccc} 
\tk {\k}{\an {\Gamma }{}}\ec Sb{}-{} & \mbox{ if } & i \not \in \sigma (\{1,\ldots,j\}) \\
{\cal J}_{0} & \mbox{ if } & i \in \sigma (\{1,\ldots,j\}) \end{array} \right. $$
For $j$ in $\{1,\ldots,k\}$, the direct sum of the ${\cal K}_{\sigma }$'s is denoted by
${\cal J}_{j,k}$ and for $\sigma $ in ${\goth S}_{1,k}$, ${\cal K}_{\sigma }$ is also 
denoted by ${\cal K}_{\sigma (1),k}$.

\begin{lemma}\label{lnor2}
Let ${\cal J}$ be the ideal of definition of $F^{(k)}$ in 
$\an {\Gamma \times {\goth b}^{k}}{}$.

{\rm (i)} The ideal $\varrho _{*}({\cal J})$ of 
$\tk {\k}{\an {\Gamma }{}}\es S{{\goth b}_{-}^{k}}$ is the sum of 
$\poi {{\cal K}}{1,k}{,\ldots,}{k,k}{}{}{}$.

{\rm (ii)} There is an exact sequence of $\an {\Gamma }{}$-modules
$$ 0 \longrightarrow {\cal J}_{k,k} \longrightarrow {\cal J}_{k-1,k} \longrightarrow 
\cdots \longrightarrow {\cal J}_{1,k} \longrightarrow \varrho _{*}({\cal J}) 
\longrightarrow 0 $$

{\rm (iii)} For $i>0$, ${\mathrm {H}}^{i}(\Gamma \times {\goth b}^{k},{\cal J})=0$ if 
${\mathrm {H}}^{i+j}(\Gamma ,{\cal J}_{0}^{\tens j})=0$ for $j=1,\ldots,k$.
\end{lemma}

\begin{proof}
(i) Let ${\cal J}_{k}$ be the sum of $\poi {{\cal K}}{1,k}{,\ldots,}{k,k}{}{}{}$. Since 
${\cal J}_{0}$ is the ideal of $\tk {\k}{\an {\Gamma }{}}\ec Sb{}-{}$ generated by 
${\cal F}_{-}$, ${\cal J}_{k}$ is a prime ideal of 
$\tk {\k}{\an {\Gamma }{}}\es S{{\goth b}_{-}^{k}}$. Moreover, ${\cal F}_{-}$ is the 
sheaf of local sections of the orthogonal complement to $F$ in 
$\Gamma \times {\goth b}_{-}$. Hence ${\cal J}_{k}$ is the ideal of definition
of $F^{(k)}$ in $\tk {\k}{\an {\Gamma }{}}\es S{{\goth b}_{-}^{k}}$, whence the assertion.

(ii) For $a$ a local section of ${\cal J}_{j,k}$ and for $\sigma $ in ${\goth S}_{j,k}$, 
denote by $a_{\sigma (1),\ldots,\sigma (j)}$ the component of $a$ on 
${\cal K}_{\sigma }$. Let $\dd $ be the map ${\cal J}_{j,k}\rightarrow {\cal J}_{j-1,k}$
such that
$$ \dd a_{\poi i1{,\ldots,}{j}{}{}{}} = \sum_{l=1}^{j} (-1)^{l+1} 
a_{\poi i1{,\ldots,}{l-1}{}{}{},\poi i{l+1}{,\ldots,}{j}{}{}{}}$$
Then by (i), we have an augmented complex
$$ 0 \longrightarrow {\cal J}_{k,k} \stackrel{\dd}\longrightarrow {\cal J}_{k-1,k} 
\stackrel{\dd}\longrightarrow \cdots 
\stackrel{\dd}\longrightarrow {\cal J}_{1,k} \longrightarrow \varrho _{*}({\cal J})
\longrightarrow 0 .$$
Let $J$ be the subbundle of the trivial bundle $\Gamma \times \ec Sb{}-{}$ such that the
fiber at $x$ is the ideal of $\ec Sb{}-{}$ generated by the fiber $F_{-,x}$ of $F_{-}$
at $x$. Then ${\cal J}_{0}$ is the sheaf of local sections of $J$ and the above augmented
complex is the sheaf of local sections of the augmented complex of vector bundles over 
$\Gamma $,
$$ 0 \longrightarrow C_{k}^{(k)}(\Gamma \times \ec Sb{}-{},J)\longrightarrow \cdots 
\longrightarrow C_{1}^{(k)}(\Gamma \times \ec Sb{}-{},J) \rightarrow J \longrightarrow 0$$
defined as in Subsection~\ref{co2}. According to Lemma~\ref{lco2} and Remark~\ref{rco2}, 
this complex is acyclic, whence the assertion by Nakayama Lemma since $J$, $\ec Sb{}-{}$ 
and the complex are graded.

(iii) Let $i$ be a positive integer such that 
${\mathrm {H}}^{i+j}(\Gamma ,{\cal J}_{0}^{\tens j})=0$ for $j=1,\ldots,k$. Then for 
$j=1,\ldots,k$ and for $\sigma $ in ${\goth S}_{j,k}$, 
${\mathrm {H}}^{i+j}(\Gamma ,{\cal K}_{\sigma })=0$ since ${\cal K}_{\sigma }$ is 
isomorphic to a sum of copies of ${\cal J}_{0}^{\tens j}$. Moreover, 
${\mathrm {H}}^{i}(\Gamma ,{\cal K}_{l,k})=0$ for $l=1,\ldots,k$ since 
${\mathrm {H}}^{i}(\Gamma ,{\cal J}_{0})=0$ by Corollary~\ref{cnor1}. Hence by (ii),
since ${\mathrm {H}}^{\bullet}$ is an exact $\delta $-functor, 
${\mathrm {H}}^{i}(\Gamma ,\varrho _{*}({\cal J}))=0$, whence the assertion since 
$\varrho $ is an affine morphism.
\end{proof}

\subsection{} \label{nor3}
For $k$ positive integer, for $j$ nonnegative integer and for 
$l=(\poi l1{,\ldots,}{k}{}{}{})$ in ${\Bbb N}^{k}$, set:
$$ {\cal M}_{j,l} := \tk {\an {\Gamma }{}}{{\cal J}_{0}^{\tens j}}\tk {\an {\Gamma }{}}
{\ex {l_{1}}{{\cal F}_{-}}}\tk {\an {\Gamma }{}}{\cdots }\ex {l_{k}}{{\cal F}_{-}}$$

\begin{lemma}\label{lnor3}
Let $k$ be a positive integer and let $(j,l)$ be in ${\Bbb N}\times {\Bbb N}^{k}$.

{\rm (i)} The $\an {\Gamma }{}$-module ${\cal J}_{0}$ is locally free.

{\rm (ii)} For $j>0$, there is an exact sequence
\begin{eqnarray*}
0 \longrightarrow \tk {\k}{\ec Sb{}-{}}{\cal M}_{j-1,(n,l)} \longrightarrow 
\tk {\k}{\ec Sb{}-{}}{\cal M}_{j-1,(n-1,l)}\longrightarrow \cdots \\ \longrightarrow 
\tk {\k}{\ec Sb{}-{}}{\cal M}_{j-1,(1,l)} \longrightarrow {\cal M}_{j,l}
\longrightarrow 0
\end{eqnarray*}

{\rm (iii)} For $i>0$, ${\mathrm {H}}^{i+j+\vert l \vert}(\Gamma ,{\cal M}_{j,l})=0$.
\end{lemma}

\begin{proof}
(i) Let $x$ be in $\Gamma $ and let $F_{-,x}$ be the fiber at $x$ of the vector bundle 
$F_{-}$ over $\Gamma $. Then $F_{-,x}$ is a subspace of dimension $n$ of ${\goth b}_{-}$.
Let $M$ be a complement to $F_{-,x}$ in ${\goth b}_{-}$. Since the map $y\mapsto F_{-,y}$
is a regular map from $\Gamma $ to $\ec {Gr}b{}-{n}$, for all $y$ in an open neighborhood 
$V$ of $x$ in $\Gamma $, 
$${\goth b}_{-} = F_{-,x} \oplus M $$
Denoting by ${\cal F}_{-,V}$ the restriction of ${\cal F}_{-}$ to $V$, we have 
$$ \tk {\k}{\an {V}{}}{\goth b}_{-} = {\cal F}_{-,V} \oplus \tk {\k}{\an V{}}M$$ 
so that 
$$ \tk {\k}{\an V{}}\ec Sb{}-{} = \tk {\k}{\ec S{}{{\cal F}}{{-,V}}{}}\es SM$$
whence 
$$ {\cal J}_{0} \left \vert _{V} \right. = \tk {\k}{\ec S{}{{\cal F}}{-,V}+}\es SM .$$
As a result, ${\cal J}_{0}$ is locally free since so is ${\cal F}_{-}$.

(ii) Since ${\cal J}_{0}$ is the ideal of $\tk {\k}{\an {\Gamma }{}}\ec Sb{}-{}$ generated
by the locally free module ${\cal F}_{-}$ of rank $n$ and since ${\cal F}_{-}$ is locally
generated by a regular sequence of the algebra $\tk {\k}{\an {\Gamma }{}}\ec Sb{}-{}$, 
having $n$ elements, we have an exact Koszul complex
$$ 0 \longrightarrow \tk {\k}{\ec Sb{}-{}}\ex n{{\cal F}_{-}} \longrightarrow \cdots 
\longrightarrow \tk {\k}{\ec Sb{}-{}}{\cal F}_{-} \longrightarrow {\cal J}_{0}
\longrightarrow 0$$
whence a complex
\begin{eqnarray*}
0 \longrightarrow \tk {\k}{\ec Sb{}-{}}
\tk {\an {\Gamma }{}}{\ex n{{\cal F}_{-}}}{\cal M}_{j-1,l}  \longrightarrow \cdots 
\longrightarrow \tk {\k}{\ec Sb{}-{}}\tk {\an {\Gamma }{}}{{\cal F}_{-}}{\cal M}_{j-1,l}
\\ \longrightarrow \tk {\an {\Gamma }{}}{{\cal J}_{0}}{\cal M}_{j-1,l}\longrightarrow 0 .
\end{eqnarray*}
According to (i), ${\cal M}_{j-1,l}$ is a locally free module. Hence this complex
is acyclic.

(iii) Prove the assertion by induction on $j$. According to Corollary~\ref{c2nor1}, 
it is true for $j=0$. Suppose that it is true for $j-1$. According to the induction 
hypothesis, for all positive integer $i$ and for $p=1,\ldots,n$, 
$${\mathrm {H}}^{i+j-1+p+\vert l \vert}(\Gamma ,\tk {\k}{\ec Sb{}-{}}{\cal M}_{j-1,(p,l)})
= \tk {\k}{\ec Sb{}-{}}{\mathrm {H}}^{i+j-1+p+\vert l \vert}(\Gamma ,{\cal M}_{j-1,(p,l)})
= 0 .$$
Then, according to (ii), ${\mathrm {H}}^{i+j+\vert l \vert}(\Gamma ,{\cal M}_{j,l})=0$ 
for all positive integer $i$ since ${\mathrm {H}}^{\bullet}$ is an exact 
$\delta $-functor.
\end{proof}

\begin{prop}\label{pnor3}
The variety ${\goth X}_{0,k}$ is Gorenstein with rational singularities and its canonical
module is free of rank $1$. Moreover the ideal of definition of ${\goth X}_{0,k}$ in 
$\ec Sb{}-{}^{\tens k}$ is the space of global sections of ${\cal J}$. 
\end{prop}

\begin{proof}
From the short exact sequence, 
$$0 \longrightarrow {\cal J} \longrightarrow \an {\Gamma \times {\goth b}^{k}}{}
\longrightarrow \an {F^{(k)}}{} \longrightarrow 0$$
we deduce the long exact sequence
$$ \cdots \longrightarrow {\mathrm {H}}^{i}(\Gamma \times {\goth b}^{k},{\cal J}) 
\longrightarrow \tk {\k}{\ec Sb{}-{}^{\tens k}}{\mathrm {H}}^{i}(\Gamma ,\an {\Gamma }{}) 
\longrightarrow {\mathrm {H}}^{i}(F^{(k)},\an {F^{(k)}}{}) \longrightarrow 
{\mathrm {H}}^{i+1}(\Gamma \times {\goth b}^{k},{\cal J}) \longrightarrow \cdots $$
According to Proposition~\ref{pxv},(i), ${\mathrm {H}}^{i}(\Gamma ,\an {\Gamma }{})=0$ 
for $i>0$ and according to Lemma~\ref{lnor2},(iii) and Lemma~\ref{lnor3},(iii),
${\mathrm {H}}^{i}(\Gamma \times {\goth b}^{k},{\cal J})=0$ for $i>0$. Hence,   
${\mathrm {H}}^{i}(F^{(k)},\an {F^{(k)}}{})=0$ for $i>0$, whence the short exact sequence
$$ 0 \longrightarrow {\mathrm {H}}^{0}(\Gamma \times {\goth b}^{k},{\cal J})
\longrightarrow \ec Sb{}-{}^{\tens k} \longrightarrow 
{\mathrm {H}}^{0}(F^{(k)},\an {F^{(k)}}{}) \longrightarrow 0$$
As $F^{(k)}$ is a desingularization of ${\goth X}_{0,k}$, 
$\k[\widetilde{{\goth X}_{0,k}}]$ is the space of global sections of $\an {F^{(k)}}{}$ by
Lemma~\ref{lint}. Then $\k[{\goth X}_{0,k}]=\k[\widetilde{{\goth X}_{0,k}}]$ since the 
image of $\ec Sb{}-{}^{\tens k}$ is contained in $\k[{\goth X}_{0,k}]$, whence the 
proposition by Proposition~\ref{prs},(ii).
\end{proof}

\begin{coro}\label{cnor3}
{\rm (i)} The normalization morphism of ${\cal C}_{\x}^{(k)}$ is a homeomorphism.

{\rm (ii)} The normalization morphism of ${\cal C}^{(k)}$ is a homeomorphism.
\end{coro}

\begin{proof}
(i) As ${\goth X}_{0,k}$ is contained in ${\goth b}^{k}$, we deduce the commutative 
diagram
$$\xymatrix{ \sqx G{{\goth X}_{0,k}} \ar@{^{(}->}[r] \ar[d] & \sqx G{{\goth b}^{k}} 
\ar[d]^{\gamma _{\x}} \\ {\cal C}_{\x}^{(k)} \ar@{^{(}->}[r] & {\cal B}_{\x}^{(k)}}$$
According to~\cite[Proposition 3.4]{CZ}, the normalization morphism of 
${\cal B}_{\x}^{(k)}$ is a homeomorphism. Then since $\sqx G{{\goth b}^{k}}$ is a 
desingularization of ${\cal B}_{\x}^{(k)}$, the fibers of $\gamma _{\x}$ are connected by
Zariski Main Theorem \cite[\S 9]{Mu}. Then so are the fibers of the restriction of 
$\gamma _{\x}$ to $\sqx G{{\goth X}_{0,k}}$ since $\sqx G{{\goth X}_{0,k}}$ is the 
inverse image of ${\cal C}_{\x}^{(k)}$. According to Proposition~\ref{pnor3}, 
$\sqx G{{\goth X}_{0,k}}$ is a normal variety. Moreover, the restriction of 
$\gamma _{\x}$ to $\sqx G{{\goth X}_{0,k}}$ is projective and birational, 
whence the commutative diagram
$$\xymatrix{ \sqx G{{\goth X}_{0,k}} \ar[rr]^{\widetilde{\gamma _{\x}}} 
\ar[rd]_{\gamma _{\x}} &&
\widetilde{{\cal C}_{\x}^{(k)}} \ar[ld]^{\lambda _{*,k}} \\ & {\cal C}_{\x}^{(k)} & }$$
with $\lambda _{*,k}$ the normalization morphism. For $x$ in ${\cal C}_{\x}^{(k)}$, 
$\lambda _{*,k}^{-1}(x) = \widetilde{\gamma _{\x}}(\gamma _{\x}^{-1}(x))$. Hence 
$\lambda _{*,k}$ is injective since the fibers of $\gamma _{\x}$ are connected, whence 
the assertion since $\lambda _{*,k}$ is closed as a finite morphism.

(ii) Denote again by $\eta $ the restriction of $\eta $ to ${\cal C}_{\x}^{(k)}$. We 
have a commutative diagram
$$\xymatrix{ \widetilde{{\cal C}_{\x}^{(k)}} \ar[d]_{\widetilde{\eta }}
\ar[r]^{\lambda _{*,k}} & 
{\cal C}_{\x}^{(k)}\ar[d]^{\eta } \\
\widetilde{{\cal C}^{(k)}} \ar[r]_{\lambda _{k}} & {\cal C}^{(k)}}$$
with $\lambda _{k}$ the normalization morphism. According to~\cite[Proposition 5.8]{CZ}, 
all fiber of $\eta $ or $\widetilde{\eta }$ is one single $W({\cal R})$-orbit and by (i), 
$\lambda _{*,k}$ is bijective. Hence $\lambda _{k}$ is bijective, whence the assertion 
since $\lambda _{k}$ is closed as a finite morphism.
\end{proof}

\appendix

\section{Notations} \label{no}
In this appendix, $V$ is a finite dimensional vector space. Denote by $\es SV$ and 
$\ex {}V$ respectively the symmetric and exterior algebras of $V$. For all integer $i$, 
$\sy iV$ and $\ex iV$ are the subspaces of degree $i$ for the usual gradation of $\es SV$
and $\ex {}V$ respectively. In particular, $\sy iV$ and $\ex iV$ are equal to zero for 
$i$ negative.\\ 

$\bullet$ For $l$ positive integer, denote by ${\goth S}_{l}$ the group of 
permutations of $l$ elements.

$\bullet$ For $m$ positive integer and for $l=(\poi l1{,\ldots,}{m}{}{}{})$ in 
${\Bbb N}^{m}$, set:
\begin{eqnarray*}
\vert l \vert := & \poi l1{+\cdots +}{m}{}{}{} \\
\sy lV := & \tk {\k}{\sy {l_{1}}V}\tk {\k}{\cdots }\sy {l_{m}}V \\
\ex lV := & \tk {\k}{\ex {l_{1}}V}\tk {\k}{\cdots }\ex {l_{m}}V .
\end{eqnarray*}

$\bullet$ For $k$ positive integer and for $l=(\poi l1{,\ldots,}{m}{}{}{})$ in 
${\Bbb N}^{m}$ such that $1\leq \vert l \vert \leq k$, denote by $V^{\tens k}$ the 
$k$-th tensor power of $V$ and by ${\goth S}_{l}$ the direct product
$\poi {{\goth S}}{l_{1}}{\times \cdots \times }{l_{m}}{}{}{}$. The group ${\goth S}_{l}$
has a natural action on $V^{\tens k}$ given by 
\begin{eqnarray*}
(\poi {\sigma }1{,\ldots,}{m}{}{}{}).(\poi v1{\tens \cdots \tens}{k}{}{}{}) = &
\poi v{\sigma _{1}(1)}{\tens \cdots \tens}{\sigma _{1}(l_{1})}{}{}{} 
 \tens \poi v{l_{1}+\sigma _{2}(1)}{\tens \cdots \tens}{l_{1}+\sigma _{2}(l_{2})}{}{}{}
\\ & \tens
\cdots \tens \poi v{\vert l \vert-l_{m}+\sigma _{m}(1)}{\tens \cdots \tens}
{\vert l \vert-l_{m}+\sigma _{m}(l_{m})}{}{}{}  \tens 
\poi v{\vert l \vert +1}{\tens \cdots \tens }{k}{}{}{} .
\end{eqnarray*}
The map 
$$ a \longmapsto \pi _{k,l}(a) := \prod_{j=1}^{m} \frac{1}{l_{j}!}
\sum_{\sigma \in {\goth S}_{l}} \sigma .a$$
is a projection from $V^{\tens k}$ onto $(V^{\tens k})^{{\goth S}_{l}}$. Moreover, the 
restriction to $(V^{\tens k})^{{\goth S}_{l}}$ of the canonical map from $V^{\tens k}$ to 
$\tk {\k}{\sy lV}V^{\tens (k-\vert l \vert)}$ is an isomorphism of vector spaces.
 
\section{Some complexes} \label{co}
Let $X$ be a smooth algebraic variety. For ${\cal M}$ a coherent $\an X{}$-module and for 
$k$ positive integer, denote by ${\cal M}^{\tens k}$ the $k$-th tensor power of 
${\cal M}$. According to Notations~\ref{no}, for all $l$ in ${\Bbb N}^{m}$ such that 
$\vert l \vert \leq k$, there is an action of ${\goth S}_{l}$ on ${\cal M}^{\tens k}$.
Moreover, $\sy l{{\cal M}}$ and $\ex l{{\cal M}}$ are coherent modules defined by the
same formulas as in Notations~\ref{no}.
\subsection{} \label{co1}
Let ${\cal E}$ and ${\cal M}$ be locally free $\an {X}{}$-modules.

\begin{prop}\label{pco1}
Let $i$ be a positive integer and suppose that
$$ {\mathrm {H}}^{i+j}(X,\tk {\an X{}}{{\cal E}^{\tens k}}{\cal M}) = 0$$
for all nonnegative integers $j,k$.

{\rm (i)} For all positive integers $m$ and $k$ and for all $l$ in ${\Bbb N}^{m}$ such 
that $\vert l \vert\leq k$,
$$ {\mathrm {H}}^{i}(X,\tk {\an X{}}
{\tk {\an X{}}{\sy {l}{{\cal E}}}{\cal E}^{\tens (k-\vert l \vert)}}{\cal M}) = 0 .$$

{\rm (ii)} For all positive integers $n_{1}$, $n_{2}$, $k$ and for all $(l,m)$ in 
${\Bbb N}^{n_{1}}\times {\Bbb N}^{n_{2}}$ such that $\vert l \vert+\vert m \vert\leq k$,
$$ {\mathrm {H}}^{i}(X,\tk {\an X{}}
{\tk {\an X{}}{\sy {l}{{\cal E}}}\tk {\an X{}}{\ex m{{\cal E}}}
{\cal E}^{\tens (k-\vert l \vert-\vert m \vert)}}{\cal M}) = 0 .$$
\end{prop}

\begin{proof}
(i) Since $\pi _{k,l}({\cal E}^{\tens k})$ is isomorphic to 
$\tk {\an X{}}{\sy l{{\cal E}}}{\cal E}^{\tens (k-\vert l \vert)}$ and since $\pi _{k,l}$
is a projector of ${\cal E}^{\tens k}$, 
$\tk {\an X{}}{\sy l{{\cal E}}}{\cal E}^{\tens (k-\vert l \vert)}$ is isomorphic to 
a direct factor of ${\cal E}^{\tens k}$ and $\tk {\an X{}}{\sy l{{\cal E}}}
\tk {\an X{}}{{\cal E}^{\tens (k-\vert l \vert)}}{\cal M}$ is isomorphic to 
a direct factor of $\tk {\an X{}}{{\cal E}^{\tens k}}{\cal M}$, whence the assertion.

(ii) Denoting by $\varepsilon (\sigma )$ the signature of the element $\sigma $
of the symmetric group ${\goth S}_{m}$, the map
$$\begin{array}{ccc}
{\cal E}^{\tens m} \longrightarrow {\cal E}^{\tens m} &&
a \mapsto \frac{1}{m!} \sum_{\sigma \in {\goth S}_{m}}\varepsilon (\sigma ) 
\sigma .a \end{array}$$
is a projection from ${\cal E}^{\tens m}$ onto a submodule of ${\cal E}^{\tens m}$ 
isomorphic to $\ex m{\cal E}$. So, $\ex m{{\cal E}}$ is isomorphic to a direct factor 
of ${\cal E}^{\tens m}$. Then, by induction on $m$, for $l$ in ${\Bbb N}^{m}$, 
$\ex l{{\cal E}}$ is isomorphic to a direct factor of ${\cal E}^{\tens \vert l \vert}$.
As a result, according to (i), for all positive integers $n_{1}$, $n_{2}$, $k$ and for 
all $(l,m)$ in ${\Bbb N}^{n_{1}}\times {\Bbb N}^{n_{2}}$ such that 
$\vert l \vert+\vert m \vert\leq k$, 
$\tk {\an X{}}{\tk {\an X{}}{\sy {l}{{\cal E}}}\tk {\an X{}}{\ex m{{\cal E}}}
{\cal E}^{\tens (k-\vert l \vert-\vert m \vert)}}{\cal M}$ is isomorphic to a 
direct factor of $\tk {\an X{}}{{\cal E}^{\tens k}}{\cal M}$, whence the assertion.
\end{proof}

\subsection{} \label{co2}
Let $W$ be a subspace of $V$ and set $E := V/W$. Let $C_{\bullet}^{(n)}(V,W)$, 
$n=1,2,\ldots$ be the sequence of graded spaces over ${\Bbb N}$ defined by the induction
relations:
$$ \begin{array}{ccccc}
C_{0}^{(1)}(V,W) := V && C_{1}^{(1)}(V,W) := W && C_{i}^{(1)}(V,W) := 0 
\end{array}$$
$$\begin{array}{ccc}
C_{0}^{(n)}(V,W) := V^{\tens n} && C_{j}^{(n)}(V,W) := 
\tk {\k}{C_{j}^{(n-1)}(V,W)}V \oplus \tk {\k}{C_{j-1}^{(n-1)}(V,W)}W 
\end{array} $$
for $i\geq 2$ and $j\geq 1$. 

\begin{lemma}\label{lco2}
Let $n$ be a positive integer. There exists a graded differential of degree $-1$ on 
$C_{\bullet}^{(n)}(V,W)$ such that the complex so defined has no homology in positive
degree.
\end{lemma}

\begin{proof}
Prove the lemma by induction on $n$. For $n=1$, $\dd $ is given by the inclusion map
$\xymatrix{W \ar[r]&V}$. Suppose that $C_{\bullet}^{(n-1)}(V,W)$ has a differential 
$\dd $ verifying the conditions of the lemma. For $j>0$, denote by $\delta $ the 
linear map
$$\xymatrix{C_{j}^{(n)}(V,W) \ar[rr] && C_{j-1}^{(n)}(V,W)}, \qquad 
(a\tens v , b\tens w) \longmapsto (\dd a \tens v + (-1)^{j}b\tens w,\dd b \tens w) $$
with $a,b,v,w$ in $C_{j}^{(n-1)}(V,W)$, $C_{j-1}^{(n-1)}(V,W)$, $V$, $W$ respectively.
Then $\delta $ is a graded differential of degree $-1$. Let $c$ be a cycle of positive
degree $j$ of $C_{\bullet}^{(n)}(V,W)$. Then $c$ has an expansion
$$ c = (\sum_{i=1}^{d} a_{i}\tens v_{i},\sum_{i=1}^{d'} b_{i}\tens v_{i})$$  
with $\poi v1{,\ldots,}{d}{}{}{}$ a basis of $V$ such that $\poi v1{,\ldots,}{d'}{}{}{}$
is a basis of $W$ and with $\poi a1{,\ldots,}{d}{}{}{}$ and $\poi b1{,\ldots,}{d'}{}{}{}$
in $C_{j}^{(n-1)}(V,W)$ and $C_{j-1}^{(n-1)}(V,W)$ respectively. Since $c$ is a cycle,
$$ \sum_{i=1}^{d} \dd a_{i}\tens v_{i} + (-1)^{j} \sum_{i=1}^{d'} b_{i}\tens v_{i} = 0$$
Hence $b_{i} = (-1)^{j+1}\dd a_{i}$ for $i=1,\ldots,d'$ so that 
$$ c + \delta (0,\sum_{i=1}^{d'} (-1)^{j} a_{i}\tens v_{i}) = 
(\sum_{i=1}^{d} a_{i}\tens v_{i} + \sum_{i=1}^{d'}  a_{i}\tens v_{i},
\sum_{i=1}^{d'} (b_{i}\tens v_{i} + (-1)^{j}\dd a_{i}\tens v_{i}) ) = 
(\sum_{i=1}^{d} a_{i}\tens v_{i} + \sum_{i=1}^{d'} a_{i}\tens v_{i},0) .$$
So we can suppose $\poi b1{ = \cdots =}{d'}{}{}{} = 0$. Then 
$\poi a1{,\ldots,}{d}{}{}{}$ are cycles of degree $j$ of $C_{\bullet}^{(n-1)}(V,W)$. 
By induction hypothesis, they are boundaries of $C_{\bullet}^{(n-1)}(V,W)$ so that 
$c$ is a boundary of $C_{\bullet}^{(n)}(V,W)$, whence the lemma.  
\end{proof}

\begin{rema}~\label{rco2}
The results of this subsection remain true for $V$ or $W$ of infinite dimension since 
a vector space is an inductive limit of finite dimensional vector spaces.
\end{rema}

\section{Rational Singularities} \label{ars}
Let $X$ be an affine irreducible normal variety. 

\begin{lemma}\label{lars}
Let $Y$ be a smooth big open subset of $X$. 

{\rm (i)} All regular differential form of top degree on $Y$ has a unique regular 
extension to $X_{\loc}$.

{\rm (ii)} Suppose that $\omega $ is a regular differential form of top degree on $Y$, 
without zero. Then the regular extension of $\omega $ to $X_{\loc}$ has no zero.
\end{lemma}

\begin{proof}
(i) Since $\Omega _{X_{\loc}}$ is a locally free module of rank one, there is an affine 
open cover $\poi O1{,\ldots,}{k}{}{}{}$ of $X_{\loc}$ such that the restriction of 
$\Omega _{X_{\loc}}$ to $O_{i}$ is a free $\an {O_{i}}{}$-module generated by some 
section $\omega _{i}$. For $i=1,\ldots,k$, set $O'_{i} := O_{i}\cap Y$. Let $\omega $ be 
a regular differential form of top degree on $Y$. For $i=1,\ldots,k$, for some regular 
function $a_{i}$ on $O'_{i}$, $a_{i}\omega _{i}$ is the restriction of $\omega $ to 
$O'_{i}$. As $Y$ is a big open subset of $X$, $O'_{i}$ is a big open subset of $O_{i}$. 
Hence $a_{i}$ has a regular extension to $O_{i}$ since $O_{i}$ is normal. Denoting again 
by $a_{i}$ this extension, for $1\leq i,j\leq k$, $a_{i}\omega _{i}$ and 
$a_{j}\omega _{j}$ have the same restriction to $O'_{i}\cap O'_{j}$ and $O_{i}\cap O_{j}$ since $\Omega _{X_{\loc}}$ is torsion free as a locally free module. Let $\omega '$ be 
the global section of $\Omega _{X_{\loc}}$ extending the $a_{i}\omega _{i}$'s. Then 
$\omega '$ is a regular extension of $\omega $ to $X_{\loc}$ and this extension is unique
since $Y$ is dense in $X_{\loc}$ and $\Omega _{X_{\loc}}$ is torsion free.

(ii)  Suppose that $\omega $ has no zero. Let $\Sigma $ be the nullvariety of $\omega '$ 
in $X_{\loc}$. If it is not empty, $\Sigma $ has codimension $1$ in $X_{\loc}$. As $Y$ is
a big open subset of $X$, $\Sigma \cap X_{\loc}$ is not empty if so is $\Sigma $. As a 
result, $\Sigma $ is empty.
\end{proof}

Denote by $\iota $ the inclusion morphism $\xymatrix{X_{\loc} \ar[r] & X}$.

\begin{lemma}\label{l2ars}
Suppose that $\Omega _{X_{\loc}}$ has a global section $\omega $ without zero. Then 
the $\an X{}$-module $\iota _{*}(\Omega _{X_{\loc}})$ is free of rank $1$. 
More precisely, the morphism $\theta $:
$$ \xymatrix{ \an X{} \ar[rr]^{\theta } && \iota _{*}(\Omega _{X_{\loc}})}, \qquad
\psi \longmapsto \psi \omega $$
is an isomorphism.
\end{lemma}

\begin{proof}
For $\varphi $ a local section of $\iota _{*}(\Omega _{X_{\loc}})$ above the open subset 
$U$ of $X$, for some regular function $\psi $ on $U\cap X_{\loc}$, 
$$ \psi (\omega \left \vert _{U\cap X_{\loc}} \right. ) = \varphi .$$
Since $X$ is normal, so is $U$ and $U\cap X_{\loc}$ is a big open subset of $U$. Hence 
$\psi $ has a regular extension to $U$. As a result, there exists a well defined morphism
from $\iota _{*}(\Omega _{X_{\loc}})$ to $\an X{}$ whose inverse is $\theta $.
\end{proof}

According to \cite{Hi}, $X$ has a desingularization $Z$ with morphism $\tau $ such that 
the restriction of $\tau $ to $\tau ^{-1}(X_{\loc})$ is an isomorphism onto $X_{\loc}$.
Since $Z$ and $X$ are varieties over $\k$, we have the commutative diagram
$$\xymatrix{ Z \ar[rr]^{\tau } \ar[rd]_{p} &&
X \ar[ld]^{q} \\ & {\mathrm {Spec}}(\k) & } .$$
According to ~\cite[V. \S 10.2]{Ha0}, $p^{!}(\k)$ and $q^{!}(\k)$ are dualizing complexes
over $Z$ and $X$ respectively. Furthermore, by ~\cite[VII, 3.4]{Ha0} or 
\cite[4.3,(ii)]{Hin}, $p^{!}(\k)[-\dim Z]$ equals $\Omega _{Z}$. Set 
$\mathpzc{D} := q^{!}(\k)[-\dim Z]$ so that 
$\tau ^{!}(\mathpzc{D})=\Omega _{Z}$ by ~\cite[VII, 3.4]{Ha0} or \cite[4.3,(iv)]{Hin}.
In particular, $\mathpzc{D}$ is dualizing over $X$.

\begin{lemma}\label{l3ars}
Suppose that $X$ has rational singularities. Let ${\cal M}$ be the cohomology in degree
$0$ of $\mathpzc{D}$. Then the $\an X{}$-modules $\tau _{*}(\Omega _{Z})$ and 
${\cal M}$ are isomorphic. In particular, $\tau _{*}(\Omega _{Z})$ has finite injective
dimension.
\end{lemma}

\begin{proof}
Since $\tau $ is a projective morphism, we have the isomorphism 
\begin{eqnarray}\label{eqars}
\xymatrix{{\mathrm {R}}\tau _{*}({\mathrm {R}}\hhom_{Z}(\Omega _{Z},\Omega _{Z}))
\ar[rr] &&  {\mathrm {R}}\hhom_{X}
({\mathrm {R}}(\tau )_{*}(\Omega _{Z}),\mathpzc{D})}
\end{eqnarray}
by ~\cite[VII, 3.4]{Ha0} or \cite[4.3,(iii)]{Hin}. 
Since ${\mathrm {H}}^{i}({\mathrm {R}}\hhom_{Z}(\Omega _{Z},\Omega _{Z}))=\an Z{}$ for
$i=0$ and $0$ for $i>0$, the left hand side of (\ref{eqars}) can be identified with
${\mathrm {R}\tau _{*}}(\an Z{})$. Since $X$ has rational singularities, 
${\mathrm {R}\tau _{*}}(\an Z{})=\an X{}$ and $\mathpzc{D}$ has only cohomology 
in degree. Moreover, by Grauert-Riemenschneider Theorem \cite{GR}, 
${\mathrm {R}}\tau _{*}(\Omega _{Z})$ has only cohomology in degree $0$, whence 
$R\tau _{*}(\Omega _{Z}) = \tau _{*}(\Omega _{Z})$. Then, by (\ref{eqars}), we have
the isomorphism 
$$ \xymatrix{ \an X{} \ar[rr] &&  \hhom_{X} ((\tau )_{*}(\Omega _{Z}),{\cal M})} .$$
As $\mathpzc{D}$ is dualizing, we have the isomorphism 
$$ \xymatrix{ R\tau _{*}(\Omega _{Z}) \ar[rr] && 
{\mathrm {R}}\hhom_{X}({\mathrm {R}}\hhom_{X}(R\tau _{*}(\Omega _{Z}),\mathpzc{D}),
\mathpzc{D})}$$
whence the isomorphism $\xymatrix{ \tau _{*}(\Omega _{Z}) \ar[r] & {\cal M}}$. As a 
result, $\tau _{*}(\Omega _{Z})$ has finite injective dimension since so has 
${\cal M}$. 
\end{proof}

\section{About singularities} \label{si}
In this section we recall a well known result. Let $X$ be a variety and $Y$ a fiber 
bundle over $X$. Denote by $\tau $ the bundle projection.

\begin{lemma}\label{lsi}
{\rm (i)} If $X$ is Gorenstein and the fibers of $\tau $ are Gorenstein, then so is
$Y$.

{\rm (ii)} If $Y$ is a Gorenstein vector bundle over $X$, then $X$ is Gorenstein.

{\rm (iii)} Suppose that $X$ and the fibers of $\tau $ have rational singularities. Then 
so has $Y$.

{\rm (iv)} If $Y$ is a vector bundle over $X$, $X$ has rational singularities if 
and only if so has $Y$.
\end{lemma}

\begin{proof}
Let $y$ be in $Y$, $x:= \tau (y)$ and $F_{x}$ the fiber of $Y$ at $x$. Denote by 
$\widehat{\an Xx{}}$ and $\widehat{\an Yy{}}$ the completions of the local
rings $\an Xx$ and $\an Yy$ respectively.

(i) By hypothesis, $\an Xx$ and $\an {F_{x}}y$ are Gorenstein. Then so is 
$\tk {\k}{\an Xx}\an {F_{x}}y$. So by~\cite[Proposition 3.1.19,(a)]{Br}, $\an Yy$ is
Gorenstein, whence the assertion.

(ii) Since $Y$ is a vector bundle over $X$, $\widehat{\an Yy}$ is a ring of formal series
over $\widehat{\an Xx}$. By~\cite[Proposition 3.1.19,(c)]{Br}, $\widehat{\an Yy{}}$ is
Gorenstein. So, by~\cite[Proposition 3.1.19,(b)]{Br}, $\widehat{\an Xx}$ is Gorenstein. 
Then by \cite[Proposition 3.1.19,(c)]{Br}, $\an Xx$ is Gorenstein, whence the assertion.

(iii) There exists a cover of $X$ by open subsets $O$ such that $\tau ^{-1}(O)$ is
isomorphic to $O\times F$. According to the hypothesis, $O$ and $F$ have rational 
singularities. Then so has $\tau ^{-1}(O)$, whence the assertion since a variety has 
rational singularities if and only it has a cover by open subsets having rational
singularities.

(iv) If $Y$ is a vector bundle over $X$, then there exists a cover of $X$ by open subsets
$O$, such that $\tau ^{-1}(O)$ is isomorphic to $O\times \k^{m}$ with 
$m=\dim Y - \dim X$. According to~\cite[p.50]{KK}, $O\times \k^{m}$ has rational 
singularities if and only if so has $O$, whence the assertion since a variety has 
rational singularities if and only it has a cover by open subsets having rational
singularities.
\end{proof}

\end{document}